\documentclass[12pt,oneside,english,british]{amsart}
\usepackage{mathptmx}
\usepackage[T1]{fontenc}
\usepackage[latin9]{inputenc}
\usepackage[a4paper]{geometry}
\usepackage{babel}
\usepackage{amsthm}
\usepackage{amstext}
\usepackage{amssymb}
\usepackage{graphicx}
\usepackage[unicode=true,pdfusetitle,
 bookmarks=true,bookmarksnumbered=false,bookmarksopen=false,
 breaklinks=false,pdfborder={0 0 1},backref=false,colorlinks=false]
 {hyperref}

\makeatletter

\newcommand{\noun}[1]{\textsc{#1}}

\numberwithin{equation}{section}
\numberwithin{figure}{section}
  \theoremstyle{plain}
  \newtheorem*{thm*}{\protect\theoremname}
  \theoremstyle{remark}
  \newtheorem{rem}{\protect\remarkname}
  \theoremstyle{definition}
  \newtheorem{defn}{\protect\definitionname}
 \theoremstyle{definition}
  \newtheorem{example}{\protect\examplename}
  \theoremstyle{plain}
  \newtheorem{lem}{\protect\lemmaname}
  \theoremstyle{plain}
  \newtheorem{cor}{\protect\corollaryname}
  \theoremstyle{remark}
  \newtheorem*{claim*}{\protect\claimname}
\theoremstyle{plain}
\newtheorem{thm}{\protect\theoremname}
  \theoremstyle{definition}
  \newtheorem{problem}{\protect\problemname}

\@ifundefined{date}{}{\date{}}
\usepackage{amsmath,amssymb}

\makeatother

  \addto\captionsbritish{\renewcommand{\claimname}{Claim}}
  \addto\captionsbritish{\renewcommand{\definitionname}{Definition}}
  \addto\captionsbritish{\renewcommand{\examplename}{Example}}
  \addto\captionsbritish{\renewcommand{\lemmaname}{Lemma}}
  \addto\captionsbritish{\renewcommand{\problemname}{Problem}}
  \addto\captionsbritish{\renewcommand{\remarkname}{Remark}}
  \addto\captionsbritish{\renewcommand{\theoremname}{Theorem}}
  \addto\captionsenglish{\renewcommand{\claimname}{Claim}}
  \addto\captionsenglish{\renewcommand{\definitionname}{Definition}}
  \addto\captionsenglish{\renewcommand{\examplename}{Example}}
  \addto\captionsenglish{\renewcommand{\lemmaname}{Lemma}}
  \addto\captionsenglish{\renewcommand{\problemname}{Problem}}
  \addto\captionsenglish{\renewcommand{\remarkname}{Remark}}
  \addto\captionsenglish{\renewcommand{\theoremname}{Theorem}}
  \providecommand{\claimname}{Claim}
  \providecommand{\definitionname}{Definition}
  \providecommand{\examplename}{Example}
  \providecommand{\lemmaname}{Lemma}
  \providecommand{\problemname}{Problem}
  \providecommand{\remarkname}{Remark}
  \providecommand{\theoremname}{Theorem}
\addto\captionsbritish{\renewcommand{\corollaryname}{Corollary}}
\addto\captionsbritish{\renewcommand{\theoremname}{Theorem}}
\addto\captionsenglish{\renewcommand{\corollaryname}{Corollary}}
\addto\captionsenglish{\renewcommand{\theoremname}{Theorem}}
\providecommand{\corollaryname}{Corollary}
\providecommand{\theoremname}{Theorem}

\begin{document}

\title{\textmd{\noun{TOPOLOGICAL DETECTION OF LYAPUNOV INSTABILITY}}}

\author{pedro teixeira}
\selectlanguage{english}%
\begin{abstract}
Given an arbitrary $C^{\,0}$ flow on a manifold $M$, let $\mbox{CMin}$
be the set of its compact minimal sets, endowed with the Hausdorff
metric, and $\mathcal{S}$ the subset of those that are Lyapunov stable.
A topological characterization of the interior of $\mathcal{S}$,
the set of Lyapunov stable compact minimal sets that are away from
Lyapunov unstable ones is given, together with a description of the
dynamics around it. In particular, $\mbox{int}_{H}\mathcal{S}$ is
locally a Peano continuum (Peano curve) and each of its countably
many connected components admits a complete geodesic metric.\foreignlanguage{british}{ }

\selectlanguage{british}%
This result establishes unexpected connections between the local topology
of $\mbox{CMin}$ and the dynamics of the flow, providing criteria
for the local detection of Lyapunov instability by merely looking
at the topology of $\mbox{CMin}$. For instance, if $\mbox{CMin}$
is not locally connected at $\varLambda\in\mbox{CMin}$, then every
neighbourhood of $\varLambda$ in $M$ contains Lyapunov unstable
compact minimal sets (hence, if $\mbox{CMin}$ is nowhere locally
connected, then every neighbourhood of each compact minimal set contains
infinitely many Lyapunov unstable compact minimal sets).
\end{abstract}


\selectlanguage{british}%


\maketitle

\section{Introduction}

The comprehension of the dynamics around compact minimal sets plays
an important role in the study of flows on manifolds. Among the concepts
that are pertinent in this context, those of Lyapunov stability/instability
are fundamental both in the conservative and in the dissipative settings
\cite{LY,BI,M1,M2}. Detecting the occurrence of Lyapunov unstable
compact minimal sets in the neighbourhood of Lyapunov stable ones
is a relevant dynamical problem, first of all because in the presence
of the former, by an arbitrarily small perturbation of the phase space
coordinates of a point, one may pass from stable to unstable almost
periodic solutions.

The set $\mbox{CMin}$ of all compact minimal sets of a flow is naturally
endowed with the Hausdorff metric, thus becoming a metric space whose
``points'' are the compact minimal sets $\varLambda\in\mbox{CMin}$.
This is a topological invariant (topologically equivalent flows have
homeomorphic $\mbox{CMin}$'s). One the other hand, examples show
abundantly that completely distinct flows may also have homeomorphic
$\mbox{CMin}$'s. Nevertheless, it turns out that the mere inspection
of the local topology of $\mbox{CMin}$ at $\varLambda$ may reveal
unexpected information about the flow dynamics around $\varLambda$.
For instance, if $\mbox{CMin}$ is \emph{not }locally connected at
$\varLambda$, then every neighbourhood of $\varLambda$ in $M$ contains
Lyapunov unstable compact minimal sets. The same holds if $\varLambda$
has no compact neighbourhood in $\mbox{CMin}$. This follows immediately
from the fact that, on any flow, the interior of the set $\mathcal{S}$
of all Lyapunov stable compact minimal sets is an open subset of $\mbox{CMin}$
that is both locally compact and locally connected. It is this quite
exceptional topological structure that furnishes criteria permitting
to detect that a given $\varLambda\in\mbox{CMin}$ does \emph{not}
belong to $\mbox{int}_{H}\mathcal{S}$ or equivalently, that $\varLambda\in\mbox{cl}_{H}\mathcal{U}=\mbox{(int}_{H}\mathcal{S})^{c}$,%
\footnote{$\mathcal{U}$ is the set of Lyapunov unstable compact minimals sets
of the flow. %
} and it is readily seen that $\varLambda\in\mbox{cl}_{H}\mathcal{U}$
iff every neighbourhood of $\varLambda$ in $M$ contains Lyapunov
unstable compact minimal sets. Actually, $\mbox{int}_{H}\mathcal{S}$
is locally a Peano continuum, having the nice ``pre-geometric''
property of existence of a complete geodesic metric on each of its
countably many connected components (Corollary \hyperlink{Corollary 3}{3},
Remark \hyperlink{Remark 1.c}{1.c}). 

\begin{figure}
\hypertarget{Fig 1.1}{}

\begin{centering}
\includegraphics{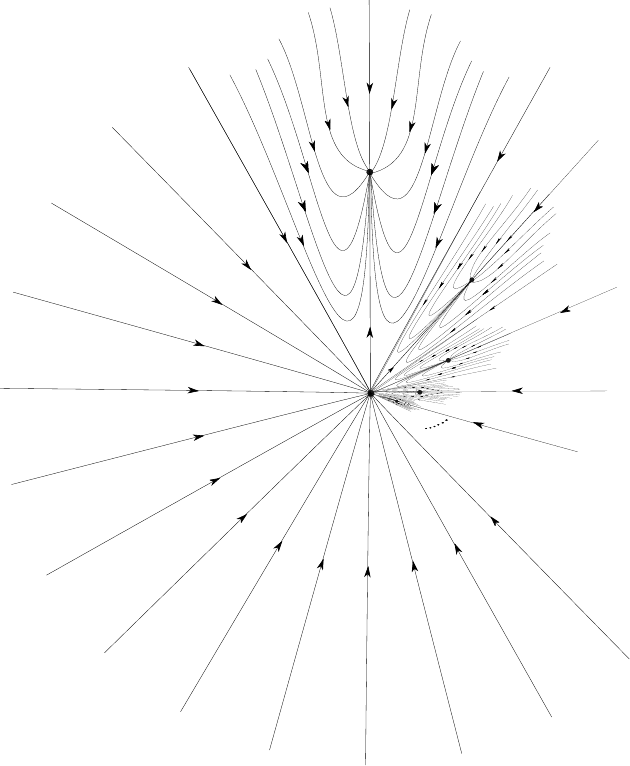}
\par\end{centering}

\caption{All equilibria, except the origin, are attractors. The $\mbox{CMin}$
of this planar flow is homeomorphic to that of (1) in Fig. 1.3.}
\end{figure}

As a practical application, imagine that without knowing a certain
flow, we are provided with a homeomorphic copy $K$ of its $\mbox{CMin}$.
If, for instance, $K$ is the Cantor star,%
\footnote{Cantor star: union of the closed radii of $\mathbb{D}^{2}$ connecting
the origin to a Cantor subset of $\mathbb{S}^{1}=\partial\mathbb{D}^{2}$.%
} then we know immediately that every neighbourhood of each compact
minimal set of the flow contains infinitely many Lyapunov unstable
compact minimal sets, and this by simply observing that $K$ is not
locally connected at a dense subset of its points. However, it is
in general impossible, by the mere inspection of $K$, to determine
\emph{which} points of $K$ correspond (under that homeomorphism)
to the Lyapunov unstable compact minimal sets we know to exist. Instead
of being a limitation, this fact is actually one of the reasons that
make these criteria interesting, for they are among the results that
somehow escape the intrinsic bounds confining the dynamical information
concerning Lyapunov stability/instability extractable from the topology
of $\mbox{CMin}$, as we now explain. To give a simpler example, consider
the $C^{\infty}$ flow on $\mathbb{R}^{2}$ pictured as (1) in Fig.
\hyperlink{Fig 1.3}{1.3} (we suppose that all orbits outside the
outer periodic orbit have that orbit as $\omega$-limit and have empty
$\alpha$-limit set; the denumerably many periodic orbits are ordered
by decreasing length as $\gamma_{n}$, $n\in\mathbb{N}$). Its $\mbox{CMin}$
is homeomorphic to $K=\{0\}\cup\{1/n:\, n\in\mathbb{N}\}$, the Lyapunov
stable equilibrium $O$ (origin) is taken to $0$ and each Lyapunov
unstable periodic orbit $\gamma_{n}$ to $1/n$. Since $K$ is not
locally connected at $0$, we know that on \emph{any} flow having
its $\mbox{CMin}$ homeomorphic to $K$, every neighbourhood of the
compact minimal set corresponding to $0$ (under that homeomorphism)
contains Lyapunov unstable compact minimal sets. Now, it is easily
seen that there is another $C^{\infty}$ flow on $\mathbb{R}^{2}$
without periodic orbits (Fig. \hyperlink{Fig 1.1}{1.1}), whose $\mbox{CMin}$
is also homeomorphic to $K$ in such a away that $0$ is taken to
the Lyapunov \emph{unstable} equilibrium $O$ and each $1/n$ is taken
to a Lyapunov \emph{stable} equilibrium (these later equilibria are
attractors i.e. asymptotically stable, see Definition \hyperlink{Definition 7}{7}).
Hence, these two flows have homeomorphic $\mbox{CMin}$'s, but under
that homeomorphism, the Lyapunov unstable compact minimal sets of
the first flow correspond to the Lyapunov stable ones of the second
and vice versa (similar examples could be given on $\mathbb{S}^{2}$).
Therefore, by the mere inspection of $K\simeq\mbox{CMin}$, it is
impossible to determine which points of $K$ correspond to the Lyapunov
unstable compact minimal sets that we know to occur in every neighbourhood
of $O$. Observe, however, that our conclusion remains intact: in
both flows every neighbourhood of the equilibrium $O$ (corresponding
to $0$ under both homeomorphisms) contains Lyapunov unstable compact
minimal sets $\varGamma$ (in the second flow, the only such $\varGamma$
is the equilibrium orbit $\{O\}$ itself).

It should be also mentioned that contrariwise, it is obviously impossible
to detect the occurrence of Lyapunov stable compact minimal sets in
a flow by merely looking at its $\mbox{CMin}$, for given any $C^{\,0\leq r\leq\infty,\omega}$
flow on a manifold $M$, there is a $C^{r}$ flow on $M\times\mathbb{S}^{1}$
with all compact minimal sets Lyapunov unstable, whose $\mbox{CMin}$
is isometric to that of the original flow on $M$. Actually, the above
mentioned topological characterization of $\mbox{int}_{H}\mathcal{S}$
is quite exceptional, for among the subsets of $\mbox{CMin}$ directly
related to its partition into Lyapunov stable and Lyapunov unstable
compact minimal sets ($\mbox{CMin}=\mathcal{S\,}\sqcup\mathcal{\, U}$),
only $\mbox{int}_{H}\mathcal{S}$ has a nice local topology (even
inducing a pre-geometric structure).%
\footnote{For instance, every $\mathbb{S}^{n\geq1}$ carries a $C^{\infty}$
flow with $\mathcal{S}$ and $\mathcal{U}$ both nowhere locally compact
and nowhere locally connected (always in relation to the Hausdorff
metric). It is also easily seen that every $n$-dimensional compact
metric space $K$ (even if nowhere locally connected) is homeomorphic
to the set $\mathcal{S}$ of all Lyapunov stable compact minimals
of some $C^{\infty}$ flow on $\mathbb{R}^{2n+2}$. The same is true
for $\mbox{bd}_{H}\mathcal{S}$, $\mathcal{S}\cap\mbox{bd}_{H}\mathcal{S}$,
$\mbox{cl}_{H}\mathcal{S}$, $\mbox{\ensuremath{\big(}cl}_{H}\mathcal{S}\big)\setminus\mathcal{U}$,
$\mathcal{U}$, $\mbox{int}_{H}\mathcal{U},$ $\mathcal{U}\cap\mbox{bd}_{H}\mathcal{U}$,
$\mbox{cl}_{H}\mathcal{\mathcal{U}}$, $\mbox{\ensuremath{\big(}cl}_{H}\mathcal{\mathcal{U}}\big)\setminus\mathcal{S}$
(obviously, $\mbox{int}_{H}$, $\mbox{bd}_{H}$, $\mbox{cl}_{H}$
stand for interior boundary and closure of subsets of $\mbox{CMin }$in
relation to the Hausdorff metric and $\mathcal{U}=\mathcal{S}^{c}:=\mbox{CMin\ensuremath{\setminus}}\mathcal{S}$).
These examples already give an idea of how topologically arbitrary
these sets may be in comparison with $\mbox{int}_{H}\mathcal{S}$.%
} 

This work continues the line of research initiated in \cite{TE} aiming
to illuminate the connections between the local topology of $\mbox{CMin}$
\foreignlanguage{english}{and the dynamics of the flow, here focused
in Lyapunov stability/instability phenomena. In the conservative setting
it gives, for instance, a somewhat purely topological counterpart
to certain well known dynamical facts, usually detected by analytic
methods in higher regularity (for instance, the occurrence of Lyapunov
unstable periodic orbits arbitrarily close to a generic elliptic periodic
orbit in dimension 3 (``\emph{...de sorte qu'il est par essence impossible
de séparer les processus stables et instables'' }in Zehnder \cite{Z2},
see also \cite{Z1,NE,M1} and} Fig. 8.3-3 ``VAK'' in \cite[p.585]{AB}).
Based on the local topological characterization of $\mbox{int}_{H}\mathcal{S}$
mentioned above, the following result establishes a conservative setting
picture of the possible interplay between Lyapunov stability and instability
in the neighbourhood of an arbitrary compact minimal set:

\begin{figure}
\begin{centering}
\includegraphics[scale=0.8]{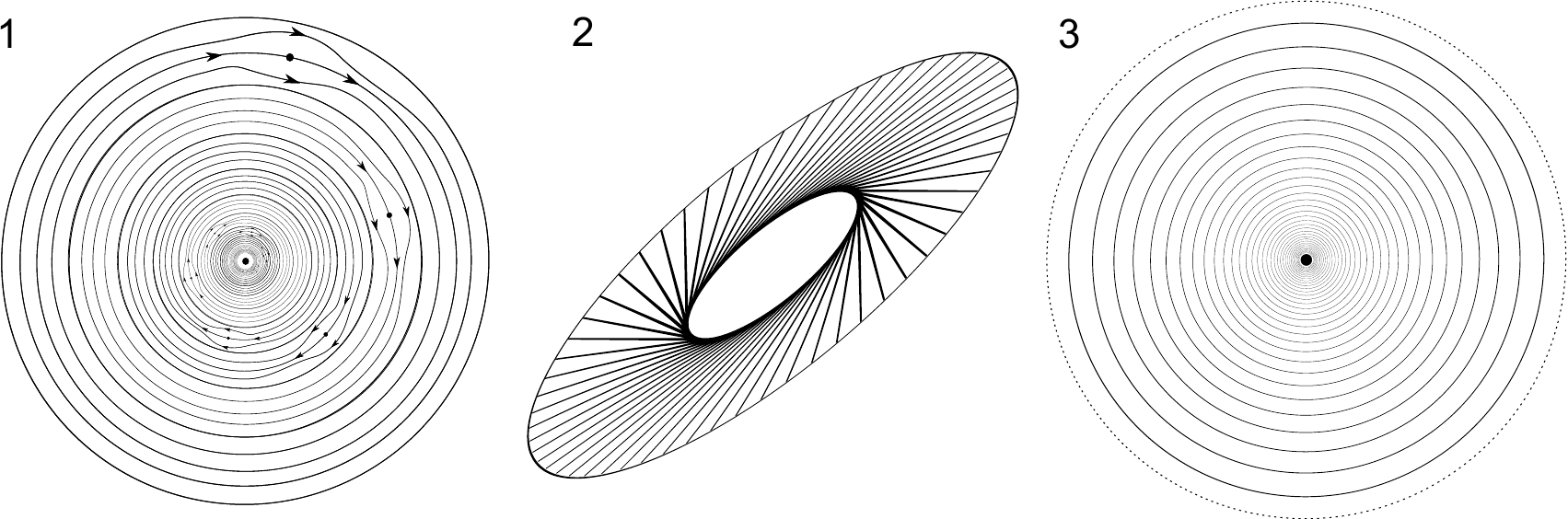}
\par\end{centering}

\caption{{\small{Theorem A, cases 1, 2 and 3.}}}
\end{figure}

\begin{thm*}
\textbf{\emph{(A) }}Let $\varLambda$ be a compact minimal set of
a $C^{0}$ non-wandering flow on a connected manifold $M$. Then,
either
\begin{enumerate}
\item there is a sequence of Lyapunov unstable compact minimal sets $\varLambda_{n}$
converging to $\varLambda$ in the Hausdorff metric, or
\item $\varLambda=M$ i.e. $M$ is compact and the flow is minimal, or
\item there are arbitrarily small compact, connected, invariant neighbourhoods
$U$ of $\varLambda$ in $M$ such that:

\begin{enumerate}
\item $U$ is the union of $\mathfrak{c}=|\mathbb{R}|$ Lyapunov stable
(compact) minimal sets $\varLambda_{i\in\mathbb{R}}$;
\item endowed with the Hausdorff metric, the set $\big\{\varLambda_{i\in\mathbb{R}}\big\}$
is a Peano continuum (Peano curve).
\end{enumerate}
\end{enumerate}
\end{thm*}
For general $C^{0}$ flows (not necessarily non-wandering), proper
attractors may show up and we have the following analogue local characterization
(Section \hyperlink{Section 4.1.2}{4.1.2}, Theorem \hyperlink{Theorem 2}{2}
and Fig. \hyperlink{Fig 1.3}{1.3}):

\begin{thm*}
\textbf{\emph{(B)}} Let $\varLambda$ be a compact minimal set of
a $C^{0}$ flow on a connected manifold $M$. Then, either
\begin{enumerate}
\item there is a sequence of Lyapunov unstable compact minimal sets $\varLambda_{n}$
converging to $\varLambda$ in the Hausdorff metric, or
\item \emph{$\varLambda$ }is an attractor i.e. asymptotically stable, or 
\item there are arbitrarily small compact, connected, (+)invariant neighbourhoods
$U$ of $\varLambda$ in $M$ such that:

\begin{enumerate}
\item the (compact) minimal sets contained in $U$ are all Lyapunov stable
and, endowed with the Hausdorff metric, their set is a Peano continuum
with $\mathfrak{c}$ elements;
\item for each $x\in U$, $\omega(x)$ is a (Lyapunov stable compact) minimal
set contained in $U$ and if $x\not\in\omega(x)$, then its negative
orbit leaves $U$ (and thus never returns again).
\end{enumerate}
\end{enumerate}
\end{thm*}
\begin{figure}
\hypertarget{Fig 1.3}{}

\begin{centering}
\includegraphics[scale=0.9]{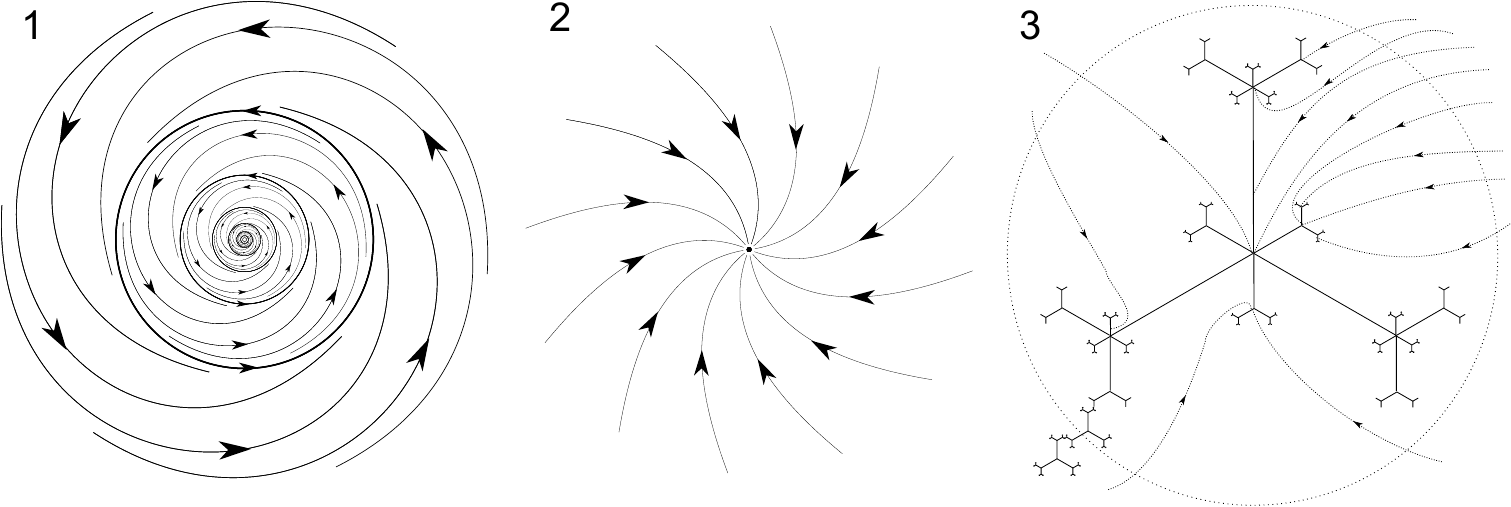}
\par\end{centering}

\caption{{\small{Theorem B, cases 1, 2 and 3.}}}
\end{figure}

This result shows that, if a compact minimal set $\varLambda$ is
away from Lyapunov unstable ones, then the dynamics around it is reasonably
well understood and the set of Lyapunov stable compact minimals near
$\varLambda$ has a remarkable topological structure (in relation
to the Hausdorff metric $d_{H}$), exactly as in the non-wandering
case (Theorem A). 

Although we have assumed the phase space $M$ to be a connected manifold,
these results still hold under much weaker hypothesis: it is enough
to suppose that $M$ is a generalized Peano continuum i.e. a locally
compact, connected and locally connected metric space (see Remark
\hyperlink{Remark 1.c}{1.c}). What seems remarkable is that \emph{this
topological structure of the phase space is actually completely inherited
by each component of} \foreignlanguage{english}{$H\mathcal{S}:=\mbox{int}_{H}\mathcal{S}$,
the set of Lyapunov stable compact minimal sets that are away (in
the Hausdorff metric) from Lyapunov unstable ones (see Definitions
}\hyperlink{Definition 3}{3}\foreignlanguage{english}{ and }\hyperlink{Definition 4}{4}\foreignlanguage{english}{).
Globally, what happens, for arbitrarily flows, is roughly the following
(Theorem }\hyperlink{Theorem 1}{1}):
\begin{enumerate}
\item $H\mathcal{S}$\emph{ }has countably many components $X_{i}$ (possibly
none), each $X_{i}$ being a clopen generalized Peano continuum (\emph{$H\mathcal{S}$
}is endowed with the Hausdorff metric);
\item the union $X_{i}^{*}\subset M$ of the (Lyapunov stable) compact minimal
sets $\varLambda\in X_{i}$ is contained in the ``basin'' $A_{i}$,
a connected, open invariant subset of $M$, consisting of all points
that have some $\varLambda\in X_{i}$ as $\omega$-limit set. Although
$X_{i}^{*}$ may be noncompact, it roughly acts as an attractor with
basin $A_{i}$ in the flow (see Fig. \hyperlink{Fig 1.4}{1.4});
\item if $x\in A_{i}$ but $x\not\in\omega(x)$, then\emph{ }$\alpha(x)\subset\mbox{bd\,}A_{i}$.
\end{enumerate}
\begin{figure}
\hypertarget{Fig 1.4}{}

\begin{centering}
\includegraphics[scale=1.3]{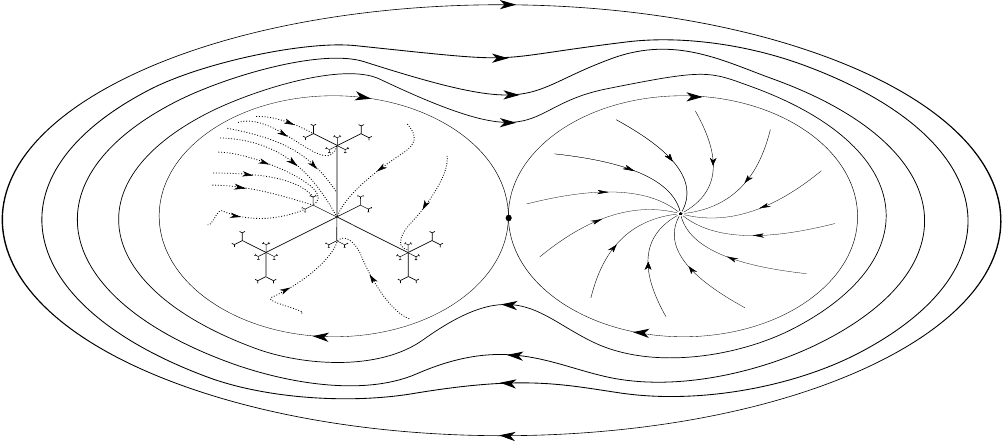}
\par\end{centering}

\caption{{\small{Flow on $\mathbb{S}^{2}$. $H\mathcal{S}$ has three components.}}}
\end{figure}

This gives a fairly complete description of the dynamics within and
around each $X_{i}^{*}$. If the flow is non-wandering but not minimal,
then $X_{i}^{*}=A_{i}$ (Theorem \hyperlink{Theorem 3}{3}), i.e.
the union of the compact minimal sets belonging to a component $X_{i}$
of $H\mathcal{S}$ is a (nonvoid) connected, open invariant set $A_{i}\subset M$
and the local density of these minimal sets is actually $\mathfrak{c}$
all over $A_{i}$: every open set $B\subset A_{i}$ intersects $\mathfrak{c}$
(Lyapunov stable) compact minimal sets $\varLambda\in X_{i}$.

The proofs explore the local topology of the phase space $M$, together
with the specific dynamical constraints of the flow near $\varLambda\in H\mathcal{S}$
and the fact that these minimal sets $\varLambda$ are continua (see
Remark \hyperlink{Remark 1.a}{1.a}), to ``move'' the local topological
structure of $M$ to the components of $H\mathcal{S}$. A kind of
duality builds up between these two topologies, Lemma \hyperlink{Lemma 3}{3}
being a simple ``show-case'' of the techniques that enable this
``crossing of the bridge''. Showing that $H\mathcal{S}$ is locally
a Peano continuum (Corollary \hyperlink{Corollary 3}{3}) requires
a fundamental result from Peano continuum theory (see the proof of
Lemma \hyperlink{Lemma 5}{5}), and seems hard to establish otherwise.

This paper is organized as follows: Section \hyperlink{Section 2}{2}
introduces the general setting of the whole work and the main concepts,
Lyapunov stability and hyper-stability (\hyperlink{Section 2.1}{2.1},
\hyperlink{Section 2.2}{2.2} and \hyperlink{Section 2.3}{2.3}).
Examples illustrating the dynamical significance of both notions are
given in \hyperlink{Section 2.4}{2.4} and \hyperlink{Section 2.5}{2.5},
the later section being entirely devoted to flows with all orbits
periodic. Section \hyperlink{Section 3}{3} introduces the main tools
and the first dynamical consequences of the topological characterization
of $H\mathcal{S}$ there obtained, criteria for the local detection
of Lyapunov unstable compact minimal sets being given in \hyperlink{Sect 3.1}{3.1}.
Section \hyperlink{Section 4}{4} contains the main results, giving
a reasonable global and local characterization of $H\mathcal{S}$
and of the dynamics around it. The analogue characterizations, specific
to the non-wandering context, are obtained in \hyperlink{Section 4.2}{4.2}.
Section \hyperlink{Section 5}{5} shows how global absence of Lyapunov
unstable compact minimals imposes strong dynamical constraints on
the flow, if the phase space is compact (these constraints vanish
in the noncompact setting). Directly related to this phenomenon are
the continuous decompositions of closed manifolds into closed submanifolds.
The natural question of the existence of a manifold structure in the
associated quotient space (endowed with the Hausdorff metric) is briefly
discussed in \hyperlink{Section 5.2}{5.2}. Finally, Section \hyperlink{Section 6}{6}
shows that the topological characterization of $H\mathcal{S}$ obtained
in Corollary \hyperlink{Corollary 3}{3} (Section \hyperlink{Section 3}{3})
is optimal in the context under consideration and ends pointing to
evidence showing that intricate (generalized) Peano continua indeed
appear as the $H\mathcal{S}$ sets of smooth flows on manifolds. Sections
\hyperlink{Section 2.4}{2.4}, \hyperlink{Section 2.5}{2.5}, \hyperlink{Section 4.1.3}{4.1.3},
\hyperlink{Section 5}{5} and \hyperlink{Section 6}{6} are unessential
and may be skipped by any reader seeking a ``straight to the core''
approach. However, the first paragraph of Section \hyperlink{Section 6}{6}
and Problem \hyperlink{Hypothesis}{2} in the same section are important
to gain perspective of the significance of the main result (Theorem
\hyperlink{Theorem 1}{1}).

\hypertarget{Section 2}{}

\section{Lyapunov stability versus instability}

\hypertarget{Section 2.1}{}

\subsection{General setting}

Throughout this paper, deductions are purely topological, all results
being valid for flows on much larger classes of phase spaces than
those of manifolds. We now establish the general context of this work.
This amounts to a minimum of hypothesis needed to deduce all the results
in the paper.\medskip{}

\begin{description}
\item [{\noun{Convention}}] Except if otherwise mentioned, we will be considering
an \emph{arbitrary continuous ($C^{0}$) flow $\theta$ on a} \emph{locally
compact, connected and locally connected metric space }$M$. Such
an $M$ is called a \emph{generalized Peano continuum} (a \emph{Peano
continuum}, if compact). $M$ is \emph{non-degenerate }if $|M|>1$,
thus implying $|M|=\mathfrak{c}=|\mathbb{R}|$.
\end{description}
($M$ is connected ($\implies|M|\geq\mathfrak{c}$) and locally compact,
hence \cite[p.269]{KO} separable ($\implies|M|\leq\mathfrak{c})$,
therefore $|M|=\mathfrak{c}$).

\medskip{}

Assuming the phase space of the flow to be a generalized Peano continuum
instead of a manifold has obvious advantages, even in the differentiable
setting: the results may be applied to subflows on arbitrary (e.g.
non-manifold) connected, closed invariant subsets, provided these
are locally connected. Obviously, if, instead, these invariant subsets
are open, then only connectedness is required. On the other hand,
this more general setting helps to get rid of the additional manifold
structure that is irrelevant for the comprehension of the phenomena
under consideration, attention becoming exclusively focused on the
topological factors that are determinant to the process.
\begin{rem}
\hypertarget{Remark 1.a}{}(a) a \emph{continuum }is a compact and
connected metric space. Compact minimal sets are continua.

(b) Peano continua (or \emph{Peano curves}) are the continuous images
of $\mathbb{D}^{1}=[-1,1]$ into Hausdorff spaces (Hahn-Mazurkiewicz
Theorem). 

\hypertarget{Remark 1.c}{}(c) connected manifolds%
\footnote{An $n$-\emph{manifold} (briefly, a \emph{manifold}) is a separable
metric space, locally homeomorphic to the $n$-cell $\mathbb{D}^{n}=\{x\in\mathbb{R}^{n}:\,|x|\leq1\}$.
Thus, manifolds are 2nd countable and Hausdorff (but possibly disconnected,
noncompact, with boundary). As usual, \emph{closed manifold} means
compact and boundaryless. Except if otherwise mentioned, \emph{all
manifolds considered are assumed to be connected}.%
} (not necessarily compact or boundaryless) are generalized Peano continua
and the later share important topological properties with the former:
they are locally compact, \emph{separable, arcwise connected, locally
arcwise connected }\cite[p.131-132]{NA} and \emph{admit a complete
geodesic metric} (Tominaga and Tanaka \cite{TO}, following Bing \cite{B1,B2}).
This implies the existence of an equivalent metric $d$ for which:
(1) every two points are joined through a geodesic arc (given any
$x,y\in M$, there is an isometric embedding $\varphi:[0,d(x,y)]\hookrightarrow M$
with $\varphi(0)=x$ and $\varphi(d(x,y))=y$); (2) closed balls are
compact (Hopf-Rinow-Cohn-Vossen Theorem \cite[p.51]{BU}). Generalized
Peano continua are at the threshold of metric geometry. It is conjectured
(Busemann \cite{BN}), that if every geodesic in an $n$-dimensional
generalized Peano continuum $M$ is locally prolongable and prolongations
are unique, then $M$ is a (boundaryless) $n$-manifold. This is confirmed
in dimension $n\leq4$ and (apparently) open in higher dimensions
(see \cite{HR,BR}). However, even $\mathbb{R}^{2}$ embeddable (generalized)
Peano continua may display intricate fractal-like geometry and topology
(as geodesics can have multiple prolongations, see e.g. \cite[Chap.1]{CH},\cite[Chap.9]{SA}
for examples).\end{rem}
\begin{defn}
(\emph{neighbourhoods and distance in $M$}) $[M,d]$ denotes the
space $M$ with metric $d$. Given $x\in M$ (resp. $Y\subset M$),
$\mathcal{N}_{x}$ (resp. $\mathcal{N}_{Y}$) is the set of neighbourhoods
of $x$ (resp. of $Y$) in $M$. For $X,\, Y\subset M$, $d(X,Y):=\mbox{inf}\{d(x,y):\, x\in X,\, y\in Y\}$.
\end{defn}

The classical concept of Lyapunov stability of compact invariant sets
is now introduced. It is crucial to understand the much stronger dynamical
constraints this notion imposes in the non-wandering setting. Given
the relevance of non-wandering flows commanded by conservative dynamics,
we shall deduce, throughout this paper, from the general results,
the corresponding characterizations that are specific to this particularly
important context. 

Consider a $C^{0}$ flow on $M$ and a nonvoid, compact invariant
set $K\subset M$.
\begin{defn}
(\emph{stable / unstable})~ $K$ is \emph{(Lyapunov)} \emph{stable}
if every $U\in\mathcal{N}_{K}$ contains a positively invariant $V\in\mathcal{N}_{K}$
$(V=\mathcal{O}^{+}(V)\subset U$, \foreignlanguage{english}{where
$\mathcal{O}^{+}(V)=\cup_{x\in V}\,\mathcal{O}^{+}(x)$ and $\mathcal{O}^{+}(x)=\{x^{\, t}:\, t\geq0\}$).}
Otherwise it is \emph{unstable}. If every $U\in\mathcal{N}_{K}$ contains
a negatively invariant $V\in\mathcal{N}_{K}$, we say that $K$ is
(-)stable. $K$ is\emph{ bi-stable} if it is both stable and (-)stable.\end{defn}
\begin{rem}
\emph{Throughout this paper, every mention to stability/instability
is always in Lyapunov's sense.}

\hypertarget{Section 2.2}{}
\end{rem}

\subsection{Non-wandering setting}

\begin{figure}
\begin{centering}
\includegraphics{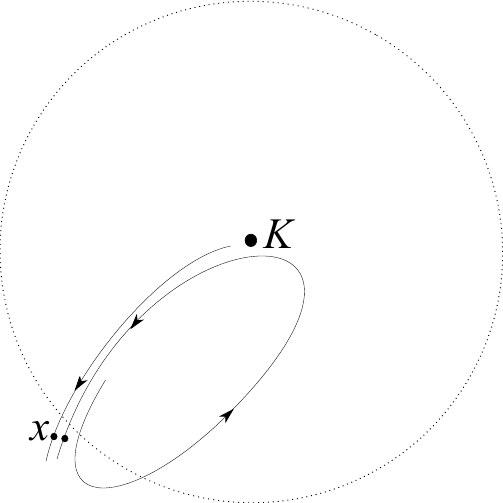}
\par\end{centering}

\caption{{\small{For non-wandering flows, $K$ unstable $\Longleftrightarrow$
$K$ (-)unstable.}}}
\end{figure}

\begin{rem}
The following elementary facts will be implicitly used throughout
this paper: (a) a nonvoid proper subset of a connected metric space
has nonvoid boundary; (b) in a locally compact metric space, compact
subsets have arbitrarily small compact neighbourhoods and both closed
and open sets are locally compact; (c) a (+)invariant set has (+)invariant
closure and interior.
\end{rem}
It is easily seen that in a\emph{ }non-wandering flow, a stable (or
(-)stable) $K$ is always bi-stable: together, the density of recurrent
points (Poincaré Recurrence Theorem)%
\footnote{Here we are invoking the following topological version of this celebrated
result: \emph{the set of recurrent points of a non-wandering $C^{0}$
flow, on a locally compact metric space $M$, is Baire residual and
thus dense in $M$ }(\emph{Proof.} for each $0\neq n\in\mathbb{Z}$,
let $R_{n}=\{x\in M:\,\exists\, t\in[1,+\infty)n\,:\, d(x^{t},x)<1/|n|\}$.
Then $R_{n}$ is open by continuity of the flow and dense in $M$
since the flow is non-wandering. Therefore, the set of recurrent points,
$R=\underset{0\neq n\in\mathbb{Z}}{\cap}R_{n}$, is Baire residual
in $M$).%
} and the continuity of the flow immediately imply that
\[
K\mbox{ \,\emph{is unstable} }\Longleftrightarrow\, K\mbox{ \,\emph{is} }\mbox{(-)}unstable\quad\mbox{(Fig. 2.1)}
\]

Therefore, in the context of\emph{ non-wandering flows}, a compact
invariant set $\emptyset\neq K\subsetneq M$ is either
\begin{enumerate}
\item \hypertarget{I}{}\emph{stable $\Longleftrightarrow$ }(-)\emph{stable
$\Longleftrightarrow$ bi-stable},\\
in which case $K$ \emph{has arbitrarily small compact invariant neighbourhoods};
due to the connectedness of $M$, this implies that \emph{for any}
$U\in\mathcal{N}_{K}$, $U\setminus K$ \emph{contains infinitely
many compact minimal sets }(\emph{Proof. }if $K$ is bi-stable and
$U_{0}\in\mathcal{N}_{K}$, then $U_{0}\setminus K$ contains a compact
minimal set: take a compact $U\in\mathcal{N}_{K}$, $M\neq U\subset U_{0}$.
As $K$ is bi-stable, there are $V_{0},V_{1}\in\mathcal{N}_{K}$ such
that $\mathcal{O}^{-}(V_{0}),\,\mathcal{O}^{+}(V_{1})\subset U$,
hence $V:=\overline{\mathcal{O}(V_{0}\cap V_{1})}\subset U$. Now
$V\neq M$ is a compact invariant neighbourhood of $K$ and $M$ is
connected, thus $\mbox{bd\,}V\subset U_{0}$ is a nonvoid compact
invariant set disjoint from $K$ and it contains at least one compact
minimal set. Therefore, since there are arbitrarily small compact
invariant neighbourhoods of $K$, $U_{0}\setminus K$ contains infinitely
many compact minimal sets), or\medskip{}

\item \emph{unstable $\Longleftrightarrow$ }(-)\emph{unstable},\\
in which case there is an $U\in\mathcal{N}_{K}$ and points $z$ arbitrarily
near $K$ such that 
\[
\mathcal{O}^{-}(z)\not\subset U\mbox{ and }\mathcal{O}^{+}(z)\not\subset U
\]

\end{enumerate}
i.e. points arbitrarily near $K$ escape from $U$ both in the past
and future. It should be mentioned that, in case (2), Ura-Kimura-Bhatia
Theorem \cite[p.114]{BH} implies that, in the absence of a pair of
points $x,\, y\in K^{c}$ such that $\,\emptyset\neq\alpha(x),\,\omega(y)\subset K$,
a kind of partial stability still takes place near $K$: there are
points $z\in K^{c}$ whose orbits remain forever (past and future)
arbitrarily near $K$ (i.e., for all $U\in\mathcal{N}_{K}$, there
are orbits $\mathcal{O}(z)\subset U\setminus K$). Obviously, both
phenomena may coexist in dimension $n\geq2$.

\hypertarget{Section 2.3}{}

\subsection{(Lyapunov) hyper-stability}

We now introduce the central concept of this paper, Lyapunov hyper-stability.
Briefly, the compact minimal sets are partitioned into Lyapunov stable
and unstable ones, the hyper-stable being those that are \emph{not}
in the (Hausdorff metric) closure of the unstable. This actually amounts
to having a neighbourhood in $M$ intersecting no (Lyapunov) unstable
compact minimal set (see below).
\begin{defn}
\hypertarget{Definition 3}{}($\mbox{CMin}$, $\mathcal{S}$, $\mathcal{U}$)
$\mbox{CMin}$ is the set of compact minimals of the flow and $\mathcal{\mathcal{S}},\,\mathcal{U}\subset\mbox{CMin}$
the subsets of those that are, respectively, (Lyapunov) stable and
unstable. For $U\subset M$, 
\[
\mbox{CMin}(U)=\mbox{set of compact minimals contained in }U
\]
$\mathcal{S}(U)$, $\mathcal{U}(U)$ are defined in an analogue way. \end{defn}
\begin{rem}
\hypertarget{Remark 4}{}$\mbox{CMin}(U)$ and more generally $\mbox{Ci}(U)$,
the set of nonvoid, compact invariant subsets of $U$ are naturally
endowed with the Hausdorff metric $d_{H}$, $\mbox{Ci}(U)$ being
compact, if $U$ is compact (see e.g. \cite[p.233]{TE}). The subscript
$_{H}$ stands for Hausdorff metric concepts, e.g. $\mbox{int}_{H}X$,
$\mbox{cl}_{H}X$ and $\mbox{bd}_{H}X$ are, respectively, the interior,
closure and boundary of $X\subset[\mbox{CMin},\, d_{H}]$.
\end{rem}
Let $\varLambda\in\mbox{CMin}$. The following three statements seem
ordered by increasing strength:
\begin{enumerate}
\item every neighbourhood of $\varLambda$ in $M$ intersects some $\varGamma\in\mathcal{U}$
\item every neighbourhood of $\varLambda$ in $M$ contains some $\varGamma\in\mathcal{U}$
\item there is a sequence $\varGamma_{n}\in\mathcal{U}$ such that $\varGamma_{n}\overset{d_{H}}{\longrightarrow}\varLambda$
\end{enumerate}
However, they are actually equivalent. $(1)\Rightarrow(2):$ this
is obvious if $\varLambda\in\mathcal{U}$, otherwise, given $U\in\mathcal{N}_{\varLambda}$,
there is a compact, (+)invariant $U\supset V\in\mathcal{N}_{\varLambda}$
(Lemma \hyperlink{Lemma 1}{1}, Remark \hyperlink{Remark 5}{5}, below).
By (1), $V$ intersects some $\varGamma\in\mathcal{U}$, hence $\varGamma\subset V\subset U$
(as $\varGamma=\overline{\mathcal{O}^{+}(x)}$, for each $x\in\varGamma$);
$(2)\Rightarrow(3):$ for each $n\geq$1, take a $\varGamma_{n}\in\mathcal{U}$
contained in $B(\varLambda,1/n).$ Then (Lemma \hyperlink{Lemma 6}{6},
below), $\varGamma_{n}\overset{d_{H}}{\longrightarrow}\varLambda$;
$(3)\Rightarrow(1):$ this is obvious since $d_{H}(\varGamma_{n},\varLambda)<\epsilon\implies\varGamma_{n}\subset B(\varLambda,\epsilon)$.
\begin{defn}
\hypertarget{Definition 4}{}(\emph{hyper-stable}, $H\mathcal{S}$)
a compact minimal set $\varLambda$ is (Lyapunov)\emph{ hyper-stable}
if some neighbourhood of $\varLambda$ in $M$ intersects \emph{no}
$\varGamma\in\mathcal{U}$. For any $U\subset M$, $H\mathcal{S}(U)$
is the set of hyper-stable $\varLambda\subset U$ and $H\mathcal{S}:=H\mathcal{S}(M)=\mbox{int}_{H}\mathcal{S}.$
Therefore, $H\mathcal{S}^{c}$ is the set of compact minimals satisfying
the three equivalent conditions (1) to (3), i.e. $H\mathcal{S}^{c}=\mbox{cl}_{H}\mathcal{U}$
or, equivalently, $H\mathcal{S}=\mbox{int}_{H}\mathcal{S}$.
\end{defn}
\hypertarget{Section 2.4}{}

\subsection{Examples}
\begin{example}
The (frictionless) mathematical pendulum, described by $\overset{.}{x}=y$,
$\overset{.}{y}=-\mbox{sin}\,2\pi x$, whose phase space is the cylinder
$M=\mathbb{S}^{1}\times\mathbb{R}$. Every orbit, except three (one
equilibrium with two homoclinic loops) is a hyper-stable compact minimal,
and they are all periodic, with the exception of the lower equilibrium
point, which is a centre. $[H\mathcal{S},d_{H}]$ is homeomorphic
to the (separated) union of $[0,1)$ and two copies of $\mathbb{R}$. 
\end{example}
\begin{figure}
\hypertarget{Fig 2.2}{}

\begin{centering}
\includegraphics{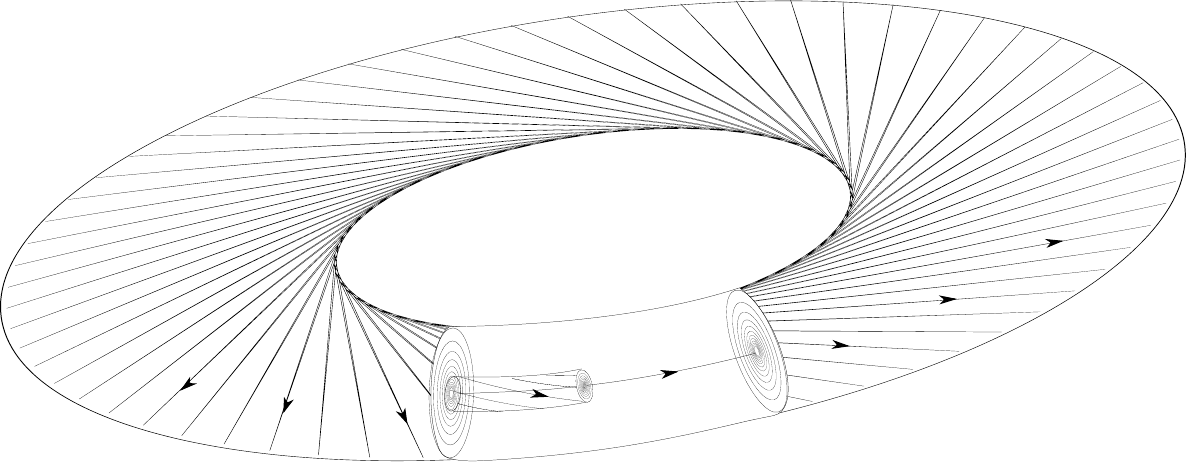}
\par\end{centering}

\caption{{\small{Example 2.}}}
\end{figure}

\begin{example}
Simple (but instructive) models of 3-dimensional volume-preserving
dynamics near a periodic orbit are given by the $C^{\infty}$ vector
fields
\[
\begin{array}{lll}
v_{\lambda}:M=\mathbb{S}^{1}\times\mathbb{D}^{2} & \longrightarrow & \mathbb{R}^{4}=\mathbb{C}^{2}\\
\qquad\,\,\,\,\,\,\,\,\,\,\,\,\,\,\,\,\,\,(z_{1},z_{2}) & \longmapsto & \big(iz_{1},i\lambda(|z_{2}|^{2})z_{2}\big)
\end{array}
\]
where $\lambda\in C^{\infty}([0,1],\mathbb{R})$ and $|\cdot|$ is
the euclidean norm (see Fig. \hyperlink{Fig 2.2}{2.2}). Each torus
$\mathbb{S}^{1}\times\mu\mathbb{S}^{1}$, $0<\mu\leq1$, is invariant
and carries either a minimal subflow ($\lambda(\mu^{2})\in\mathbb{Q}^{c}$)
or foliates into periodic orbits $(\lambda(\mu^{2})\in\mathbb{Q})$.
Thus, every point $z\in M$ belongs to some $\varLambda\in\mbox{CMin}$.
To understand the Lyapunov's nature of these minimal sets, three cases
are particularly relevant ($\simeq$ means homeomorphic):
\begin{enumerate}
\item $\lambda=const.\in\mathbb{Q}$, in which case $[H\mathcal{S},d_{H}]\simeq\mathbb{D}^{2}$.
This case defines the local fibre structure of Seifert fibrations
\cite[p.67]{E1};
\item $\lambda=const.\in\mathbb{Q}^{c}$, in which case $[H\mathcal{S},d_{H}]\simeq\mathbb{D}^{1}$;
\item $\lambda$ is constant in no nontrivial interval $I\subset[0,1]$,
in which case $H\mathcal{S}=\emptyset$ (see below).
\end{enumerate}
\end{example}
In general, $\lambda$ is constant on each component of a (possibly
empty) open set $A\subset[0,1]$, and constant in no neighbourhood
of each point $x\in A^{c}$. It follows that if $\mu^{2}\in A$, $0<\mu\leq1$,
then every $\varLambda\in\mbox{CMin}$ contained in the invariant
torus $\mathbb{S}^{1}\times\mu\mathbb{S}^{1}$ is hyper-stable, otherwise 
\begin{itemize}
\item it is an unstable periodic orbit, if $\lambda(\mu^{2})\in\mathbb{Q}$
(due to the existence of a sequence of minimal tori $\mathbb{S}^{1}\times\mu_{n}\mathbb{S}^{1}$,
$\mu_{n}\rightarrow\mu$, $\lambda(\mu_{n}^{2})\in\mathbb{Q}^{c}$,
$d_{H}$ converging to the torus $\mathbb{S}^{1}\times\mu\mathbb{S}^{1}$
containing the periodic orbit $\varLambda$);
\item $\varLambda=\mathbb{S}^{1}\times\mu\mathbb{S}^{1}$ is a stable, but
not hyper-stable minimal torus, if $\lambda(\mu^{2})\in\mathbb{Q}^{c}$
(the stability of $\varLambda$ is forced by the existence of sequences
of minimal tori $\mathbb{S}^{1}\times\xi_{n}\mathbb{S}^{1}$, $\xi_{n}\rightarrow\mu$,
$\lambda(\xi_{n}^{2})\in\mathbb{Q}^{c}$, $d_{H}$ converging to $\varLambda$
from both sides, thus entailing the existence of arbitrarily small
invariant neighbourhoods; $\varLambda$ is not hyper-stable due to
the existence of unstable periodic orbits in tori $\mathbb{S}^{1}\times\zeta_{n}\mathbb{S}^{1}$,
$\zeta_{n}\rightarrow\mu$, $\lambda(\zeta_{n}^{2})\in\mathbb{Q}$,
and thus in every neighbourhood of $\varLambda=\mathbb{S}^{1}\times\mu\mathbb{S}^{1}$).
\end{itemize}
The core periodic orbit $\gamma=\mathbb{S}^{1}\times\{0\}$ is always
stable, being hyper-stable iff $0\in A$.

\hypertarget{Section 2.5}{}

\subsection{Hyper-stability in flows with all orbits periodic}
\begin{defn}
$\mbox{Per}:=\,$the set of periodic orbits of the flow.
\end{defn}
If a flows with all orbits periodic induces a circle bundle on $M$,
then $H\mathcal{S}=\mbox{Per}$ and $[H\mathcal{S},d_{H}]$ is homeomorphic
to the base space of the bundle, which, however, may be a non-manifold
(see Section \hyperlink{Section 5.2}{5.2}).
\begin{example}
A classical example of a volume-preserving $C^{\omega}$ periodic
flow%
\footnote{A flow $\theta$ without equilibria is \emph{periodic} if there is
a $t>0$ such that $\theta^{t}=\mbox{Id}$.%
} inducing a nontrivial circle bundle is given by the Hopf flow\emph{
}on $\mathbb{S}^{2n+1}$ ($\mathbb{R}^{2n+2}$ identified with $\mathbb{C}^{n+1}$),
\[
(t,\, z)\longmapsto e^{it}z
\]
The base space of the induced bundle is $\mathbb{C}P^{n}\simeq H\mathcal{S}$,
the case $n=1$, $\mathbb{S}^{1}\hookrightarrow\mathbb{S}^{3}\rightarrow\mathbb{S}^{2}\simeq\mathbb{C}P^{1}$
being the original Hopf fibration \cite[p.230]{DU}.
\end{example}
It is possible that $H\mathcal{S}=\mbox{Per}$, even if a flows with
all orbits periodic does not induce a circle bundle. $C^{1}$ flows
with all orbits periodic on orientable, closed 3-manifolds induce
Seifert fibrations (Epstein, \cite{E1}), which in general are not
circle bundles. This guarantees that $H\mathcal{S}=\mbox{Per}$ and
also, that $[H\mathcal{S},d_{H}]$ is always a closed 2-manifold,
a remarkable phenomenon unparalleled in the higher dimensions (see
also Section \hyperlink{Section 5.2}{5.2}, Example \hyperlink{Example 4}{4}).

However, there are $C^{\infty}$ flows with all orbits periodic on
closed manifolds, for which $H\mathcal{S}\subsetneq\mbox{Per}$ i.e.
$\mathcal{U}\neq\emptyset$. Sullivan \cite{SU} constructed a beautiful
example of a such a flow (which can be made $C^{\omega}$), on a closed
5-manifold, for which there is no bound to the lengths of the orbits.
This implies (Epstein \cite[Theorem 4.3]{E2}) the existence of Lyapunov
unstable orbits in the flow (take a convergent sequence $x_{n}\in\gamma_{n}$,
where $\gamma_{n}$ is a sequence orbits with no bound on their lengths.
Then, the orbit $\mathcal{O}(\mbox{lim}\, x_{n})$ is necessarily
unstable).

\hypertarget{Section 3}{}

\section{Main lemmas. First consequences}

Lemmas \hyperlink{Lemma 1}{1} and \hyperlink{Lemma 3}{3} constitute
the core tools for the comprehension of the dynamics near hyper-stable
compact minimals. Lemma \hyperlink{Lemma 3}{3}, in particular, essentially
shows how each component of $H\mathcal{S}$ inherits the local topological
structure of the phase space (i.e. its local compactness and local
connectedness).
\begin{lem}
\hypertarget{Lemma 1}{}Every\foreignlanguage{english}{\textup{ $\varLambda\in H\mathcal{S}$}}
has arbitrarily small compact, connected, (+)invariant neighbourhoods
$U$ in $M$ such that \emph{$\mbox{CMin}(U)\subset H\mathcal{S}.$}\end{lem}
\begin{rem}
\hypertarget{Remark 5}{}Except for the the conclusion $\mbox{CMin}(U)\subset H\mathcal{S}$,
the same holds if $\varLambda\in\mathcal{S}$. \end{rem}
\begin{proof}
$M$ is locally compact and $\varLambda\in H\mathcal{S}$, hence $\mbox{\ensuremath{\varLambda}}$
\foreignlanguage{english}{has arbitrarily small compact neighbourhoods
containing no $\varGamma\in\mbox{cl}_{H}\mathcal{U}$ (Section }\hyperlink{Section 2.3}{2.3}\foreignlanguage{english}{).
Let $W$ be one of them. As $\varLambda$ is stable and $W$ is compact,
there is a $V\in\mathcal{N}_{\varLambda}$ such that $\overline{\mathcal{O}^{+}(V)}\subset W$.
Take a minimal finite cover of $\varLambda$ by sufficiently small
connected open sets $B_{i}\subset V$ ($M$ is locally connected).
Then $U:=\overline{\mathcal{O}^{+}(\cup B_{i})}\subset W$ is a compact,
connected (Remark }\hyperlink{Remark 1.a}{1.a}\foreignlanguage{english}{),
(+)invariant neighbourhood of $\varLambda$ and $\mbox{CMin}(U)\subset H\mathcal{S}=(\mbox{cl}_{H}\mathcal{U})^{c}$.}\end{proof}
\begin{lem}
\hypertarget{Lemma 2}{} Suppose that

\emph{(a)} $K,\, K\text{'}$ are nonvoid, compact invariant sets;

\emph{(b)} \emph{$\varLambda_{n}\in\mbox{CMin}$};

\emph{(c)} $\alpha(x),\,\omega(x)$ are nonvoid, compact limit sets.
Then,
\begin{enumerate}
\item $\varLambda_{n}\overset{d_{H}}{\longrightarrow}K$ and $K'\subsetneq K$
$\Longrightarrow$ $K'$ is unstable;
\item $K'\subsetneq\omega(x)$ $\Longrightarrow$ $K'$ is unstable;
\item $K\subset\alpha(x)$ is stable $\Longrightarrow$ $x\in K=\alpha(x)$;
\item if the flow is non-wandering and $\omega(x)$ is stable, then $x\in\omega(x)$.
\end{enumerate}
\end{lem}
\begin{proof}
In each case, assume that the hypothesis hold.

(1) take $y\in K\setminus K'$ and let $d=d(y,K')$. Given $\epsilon>0$,
take $n\in\mathbb{N}$ such that $d':=d_{H}(\varLambda_{n},K)<\mbox{min}(\epsilon,d/2)$.
Take $a,b\in\varLambda_{n}$ such that 
\[
\begin{array}{c}
d(a,K')\leq d'<\epsilon\mbox{ \,\ and \,}d(b,y)\leq d'<d/2\end{array}
\]
As $\varLambda_{n}$ is minimal, for some $t>0$, 
\[
d(a^{t},b)<d/2-d(b,y)
\]
thus implying that 
\[
d(a^{t},y)\leq d(a^{t},b)+d(b,y)<d/2
\]
which by its turn implies $d(a^{t},K')>d/2$. Therefore, $K'$ is
unstable, as points arbitrarily near $K'$ escape $B(K',d/2)$ in
positive time.

(2) let $y\in\omega(x)\setminus K'$. Given $\epsilon_{0}>0$, take
\[
0<\epsilon<\mbox{min}\big(\epsilon_{0},d(y,K')/2\big)
\]
Then, there are $0<t<T$ such that $d(x^{\, t},K')<\epsilon<\epsilon_{0}$
and $d(x^{T},y)<\epsilon<d(y,K')/2$, the last inequality implying
$d(x^{T},K')>d(y,K')/2$, hence $K'$ is unstable as in (1).

(3) $\emptyset\neq K\subset\alpha(x)$ and $x\not\in K$ obviously
implies that $K$ is unstable, therefore $x\in K$ and $K=\alpha(x)$
(as $K$ is a closed invariant set).

(4) $\omega(x)$ has arbitrarily small invariant neighbourhoods (Section
\hyperlink{Section 2.2}{2.2}), hence no point outside $\omega(x)$
may have $\omega(x)$ as $\omega$-limit.\end{proof}
\begin{lem}
\hypertarget{Lemma 3}{}Let $U$ be a nonvoid, compact, connected,
(+)invariant set such that \emph{$\mbox{CMin}(U)\subset H\mathcal{S}$.}
Then:
\begin{enumerate}
\item ($\omega$-limits are in $H\mathcal{S}(U)$)~ $x\in U\implies\emptyset\neq\omega(x)\in H\mathcal{S}(U)$;
\item ($\omega$-limits convergence)~ $U\ni x_{n}\rightarrow x\implies\omega(x_{n})\overset{d_{H}}{\longrightarrow}\omega(x)$;
\item (escaping orbits)~ $x\in U$ and $x\not\in\omega(x)\implies\mathcal{O}^{-}(x)\not\subset U$;
\item ($H\mathcal{S}(U)$ is a continuum)~ \emph{$H\mathcal{S}(U)$} is
$d_{H}$ compact and connected.
\end{enumerate}
\end{lem}
\begin{defn}
\hypertarget{Def 6}{}if $C\subset\mbox{CMin}$, then $C^{*}:=\underset{\varLambda\in C}{\cup}\varLambda\subset M$.\end{defn}
\begin{proof}
(1) since $U$ is compact and (+)invariant, $\emptyset\neq\omega(x)\subset U$
is compact, for every $x\in U$ and it must be a minimal set (stable
by hypothesis), otherwise there is an unstable compact minimal set
$\varLambda\subsetneq\omega(x)$ (Lemma \hyperlink{Lemma 2}{2.2}),
contradicting $\mbox{CMin}(U)\subset H\mathcal{S}$. 

(2) let $U\ni x_{n}\rightarrow x$. Since $U$ is compact and the
$\omega(x_{n})$'s are nonvoid compact invariants, by Blaschke Principle
(\cite[p.233]{TE}), there is a subsequence $\omega(x_{n'})\overset{d_{H}}{\longrightarrow}\varGamma$,
for some $\varGamma\in\mbox{Ci}(U)$ (see Remark \hyperlink{Remark 4}{4}
above). Reasoning by contradiction, suppose that $\varGamma\neq\omega(x)$.
Take $z\in\varGamma\setminus\omega(x)\neq\emptyset$ ($\omega(x)$
is minimal by (1)). Then, by continuity of the flow, there are sequences
$0<t_{n}<T_{n}$ such that 
\[
d\big(x_{n}^{t_{n}},\,\omega(x)\big)\rightarrow0\mbox{ \,\ and \,}d\big(x_{n}^{T_{n}},\, z\big)\rightarrow0
\]
which implies that $\omega(x)$ is unstable, contradicting (1). Hence,
every subsequence $\omega(x_{n'})$ of $\omega(x_{n})$ contains a
subsequence $\omega(x_{n''})\overset{d_{H}}{\longrightarrow}\omega(x)$,
trivially implying that $\omega(x_{n})\overset{d_{H}}{\longrightarrow}\omega(x)$.

(3) suppose that $x\in U$ and $x\not\in\omega(x)$. Then, $\mathcal{O}^{-}(x)\not\subset U$,
otherwise $\emptyset\neq\alpha(x)\subset U$ is compact, implying
that $\alpha(x)$ contains a compact minimal set $\varGamma$ and
$x\not\in\varGamma$ (since $x\not\in\omega(x))$, thus implying (Lemma
\hyperlink{Lemma 2}{2.3}) that $\varGamma\in\mbox{CMin}(U)$ is unstable,
contradiction. Actually, 
\[
\exists\, T\leq0:\, x^{(-\infty,T)}\subset U^{c}\mbox{ and }\; x^{[T,+\infty)}\subset U
\]
since $U$ being compact and (+)invariant, 
\[
-\infty<T:=\mbox{min}\big\{ t\leq0:\, x^{[t,0]}\subset U\big\}\leq0
\]
is well defined and has the required properties.

(4.1) $H$$\mathcal{S}(U)$ \emph{is $d_{H}$ compact}: given a sequence
$\varLambda_{n}\in H\mathcal{S}(U)=\mbox{CMin}(U)$, take a convergent
subsequence $x_{n'}\in\varLambda_{n'}$ ($U$ is compact). Then, $\varLambda_{n'}=\omega(x_{n'})$
and by (2) and (1)
\[
\omega(x_{n'})\overset{d_{H}}{\longrightarrow}\omega(\mbox{lim\,}x_{n'})\in H\mathcal{S}(U).
\]

(4.2) $H$$\mathcal{S}(U)$ \emph{is $d_{H}$ connected}: reasoning
by contradiction assume it is not. Then, by (4.1), $H\mathcal{S}(U)$
is the union of two nonvoid, disjoint compacts $C_{0},$ $C_{1}\subset H\mathcal{S}(U)$.

\hypertarget{Lemma 3.claim}{}

\emph{Claim. }$C_{0}^{*},\, C_{1}^{*}\subset U$ \emph{are nonvoid,
disjoint compacts}: they are obviously nonvoid. Fix $j\in\{0,1\}$.
Let $x_{n}\in C_{j}^{*}$. Each $x_{n}$ belongs to some $\varLambda_{n}\in C_{j}$.
Take a convergent subsequence $\varLambda_{n'}\overset{d_{H}}{\longrightarrow}\varLambda\in C_{j}$.
$U$ being compact, $x_{n'}$ has a convergent subsequence $x_{n''}\rightarrow x\in U$.
But \foreignlanguage{english}{$\varLambda_{n''}\overset{d_{H}}{\longrightarrow}\varLambda$
implies that $x\in\varLambda$, thus $x\in C_{j}^{*}$ and $C_{j}^{*}$
is compact. $C_{0}^{*}$, $C_{1}^{*}$ are disjoint since $C_{0}\cap C_{1}=\emptyset$
and distinct minimal sets are disjoint. }

\selectlanguage{english}%
Now let
\[
\varOmega_{j}:=\big\{ x\in U:\,\omega(x)\in C_{j}\big\},\,\mbox{ for \,}j=0,1
\]
We show that these two sets form a nontrivial partition of $U$ into
closed subsets, thus getting a contradiction. Since $C_{0}\cap C_{1}=\emptyset$,
$\varOmega_{0}\cap\varOmega_{1}=\emptyset$; $\varOmega_{0}\neq\emptyset\neq\varOmega_{1}$
since $C_{j}^{*}\subset\varOmega_{j}$. Also, since $C_{0},\, C_{1}\subset H\mathcal{S}(U)$
are closed, by (2), both $\varOmega_{0}$ and $\varOmega_{1}$ are
closed in $U$. Finally, by (1), $U=\varOmega_{0}\sqcup\varOmega_{1}$,
hence $U$ is disconnected.\end{proof}
\begin{lem}
\hypertarget{Lemma 4}{}\foreignlanguage{english}{$H\mathcal{S}$
is $d_{H}$ locally compact and locally connected.}\end{lem}
\selectlanguage{english}%
\begin{proof}
Let $\mathfrak{N}$ be a neighbourhood of $\varLambda$ in $H\mathcal{S}$.
By Lemmas \foreignlanguage{british}{\hyperlink{Lemma 1}{1}} and \foreignlanguage{british}{\hyperlink{Lemma 7}{7}}
(below), there is \foreignlanguage{british}{a compact, connected,
(+)invariant neighbourhood $U$ of $\varLambda$ in $M$ such that
$H\mathcal{S}(U)$ is a neighbourhood of $\varLambda$ in $H\mathcal{S}$,
contained in $\mathfrak{N}.$ By Lemma \hyperlink{Lemma 3}{3.4},
$H\mathcal{S}(U)$ is $d_{H}$ compact and connected i.e. a continuum
in the $d_{H}$ metric.}
\end{proof}
\selectlanguage{british}%

\subsection{Topological detection of Lyapunov instability}

\hypertarget{Sect 3.1}{}

We can now establish criteria permitting to detect the presence of
Lyapunov unstable compact minimal sets in arbitrarily small neighbourhoods
of a given compact minimal set $\varLambda$ i.e. criteria detecting
that $\varLambda\in\mbox{cl}_{H}\mathcal{U}$. The following result
is an immediate consequence of Lemma \hyperlink{Lemma 4}{4} (as $H\mathcal{S}^{c}=\mbox{cl}_{H}\mathcal{U}$
is a closed subset of $\mbox{CMin}$). Although trivially $1\implies$2,
we list both conditions as it is often useful to have the simplest
possible sufficient conditions in mind. 
\begin{cor}
\hypertarget{Corollary 1}{}Let \emph{$\varLambda\in\mbox{CMin}$.}
If any of the following three conditions holds,\foreignlanguage{english}{
then \emph{$\varLambda\in\mbox{cl}_{H}\mathcal{U}$} i.e. every neighbourhood
of $\varLambda$ in $M$ contains Lyapunov unstable compact minimal
sets:}
\selectlanguage{english}%
\begin{enumerate}
\item \emph{$\mbox{CMin}$ }is not locally connected at $\varLambda$;
\item $\varLambda$ has no locally connected neighbourhood in \emph{$\mbox{CMin}$};
\item $\varLambda$ has no compact neighbourhood in \emph{$\mbox{CMin}$.}
\end{enumerate}
\end{cor}
From Corollary 1 we can draw some interesting conclusions: for instance,
if $\mbox{CMin}$ is nowhere locally connected,\emph{ }then every
neighbourhood of each compact minimal set contains infinitely many
Lyapunov unstable compact minimals. This follows observing that nowhere
locally connected sets have no isolated elements. Actually, we have
the following stronger 
\begin{description}
\item [{Criterion}] \emph{If any of the following two conditions holds,}\foreignlanguage{english}{
}\emph{then every neighbourhood of each compact minimal set of the
flow contains infinitely many Lyapunov unstable compact minimals:}

\begin{enumerate}
\item $\mbox{CMin}$ \emph{is not locally connected at a dense subset of
its points;}%
\footnote{See the Cantor star example in the introduction to this paper.%
}
\item \emph{no point of }$\mbox{CMin}$\emph{ has both a compact neighbourhood
and a locally connected neighbourhood.}
\end{enumerate}
\end{description}
\begin{defn}
\hypertarget{Definition 7}{}(\emph{attractor} / \emph{repeller})
a nonvoid, compact invariant set $\varDelta$ is an \emph{attractor
}if it is (Lyapunov) stable and 
\[
B^{+}(\varDelta):=\{x\in M:\,\emptyset\neq\omega(x)\subset\varDelta\}\in\mathcal{N}_{\varDelta}
\]
i.e. if it is asymptotically stable. $B^{+}(\varDelta)$ is the \emph{attraction
basin} of\emph{ }$\varDelta$. $\mathcal{A}\subset H\mathcal{S}$
is the set of (compact) minimal attractors. $\varDelta$ is a \emph{repeller}
if it is an attractor in the time-reversed flow. Its \emph{repulsion
basin} is defined in the obvious way. Note that, when the phase space
$M$ is compact, $M$ is always both an attractor and a repeller in
any flow. This is the only possible attractor (resp. repeller) if
the flow is non-wandering (as attractors are simultaneously stable
and isolated, see (\hyperlink{I}{1}) in Section \hyperlink{Section 2.2}{2.2}
and Definition \hyperlink{Definition 8}{8} below).
\end{defn}
Local connectedness of $H\mathcal{S}$ (Lemma \hyperlink{Lemma 4}{4}),
permits a straightforward deduction of topological-dynamical results
that otherwise seem somewhat surprising. For instance, while $\mathcal{A}\subset H\mathcal{S}$\emph{,
}hyper-stable compact minimals cannot be surrounded by attractors:\emph{ }
\begin{cor}
\hypertarget{Corollary 2}{}Let $\varLambda$ be a compact minimal
set.\foreignlanguage{english}{ If every neighbourhood of $\varLambda$
in $M$ contains an attractor (not necessarily a minimal set), then
either}
\selectlanguage{english}%
\begin{enumerate}
\item $\varLambda$ is an attractor (the unique one inside its basin), or
\item \emph{$\varLambda\in\mbox{cl}_{H}\mathcal{U}$} i.e. every neighbourhood
of $\varLambda$ in $M$ also contains Lyapunov unstable compact minimal
sets.
\end{enumerate}

Therefore, if $\varLambda\in H\mathcal{S}$ then attractors $\varGamma\neq\varLambda$
cannot occur arbitrarily near $\varLambda$. 

\end{cor}
\selectlanguage{british}%
\begin{rem}
\selectlanguage{british}%
The existence of a sequence of attractors $\mbox{\ensuremath{\varDelta}}_{n}\neq\varLambda\in\mbox{CMin}$
such that $\varDelta_{n}\subset B(\varLambda,1/n)\subset M$ thus
implies the existence of another sequence $\mathcal{U}\ni\varGamma_{n}\overset{d_{H}}{\longrightarrow}\varLambda$.\end{rem}
\begin{proof}
Assume that $\varLambda\not\in\mbox{cl}_{H}\mathcal{U}$. By Lemma
\hyperlink{Lemma 4}{4} $\varLambda$ has connected neighbourhoods
in $\mbox{CMin}$. We show that the only possible such neighbourhood
is $\{\varLambda\}$ i.e. $\varLambda$ is $d_{H}$ isolated in $\mbox{CMin}$,
hence $\varLambda$ is an attractor (see Corollary \hyperlink{Corollary 5}{5}
ahead). Let $\mathfrak{N}$ be a connected neighbourhood of $\varLambda$
in $\mbox{CMin}$. \foreignlanguage{english}{We claim that given $\epsilon>0$,
there is a compact minimal set $\varGamma$ such that $d_{H}(\varGamma,\varLambda)<\epsilon$
and the connected component of $\varGamma$ in $\mbox{CMin}$ has
$d_{H}$ diameter less than $\epsilon$. Since for $\epsilon$ sufficiently
small $\varGamma$ belongs to $\mathfrak{N}$, the $d_{H}$ diameter
of $\mathfrak{N}$ must be zero i.e. $\mathfrak{N}=\{\varLambda\}$.
To prove the claim: given $\epsilon>0$, by Lemma }\hyperlink{Lemma 6}{6}
there is an open neighbourhood $U$ of $\varLambda$ in $M$ such
that $d_{H}(\varGamma,\varLambda)<\epsilon/2$ for every compact minimal
set $\varGamma\subset U$. By hypothesis, there is an attractor $Q\subset U$
and $Q$ contains at least one $\varGamma\in\mbox{CMin}$. Clearly
the attraction basin $B(Q)$ is an open invariant set containing $Q$
such that $B(Q)\setminus Q$ intersects no compact minimal. Therefore,
all compact minimals belonging to the connected component $\varTheta$
of $\varGamma$ in $\mbox{CMin}$ must be contained in $Q$, hence
$\mbox{diam}_{H}\varTheta<\epsilon$.
\end{proof}

\subsection{$H\mathcal{S}$ is locally a Peano continuum}

\selectlanguage{english}%
One of our main goals is to prove that $H\mathcal{S}$ is locally
a Peano continuum. Observe that Lemma \foreignlanguage{british}{\hyperlink{Lemma 4}{4}}
does not prove this i.e. it does not show that each $\varLambda\in H\mathcal{S}$
has arbitrarily small compact, connected and locally connected neighbourhoods
in $H\mathcal{S}$ (in relation to the Hausdorff metric $d_{H}$).
While, by Lemma \foreignlanguage{british}{\hyperlink{Lemma 1}{1}},
$\varLambda$ has arbitrarily small compact, connected, (+)invariant
neighbourhoods $U$ in $M$ such that $\mbox{CMin}(U)\subset H\mathcal{S}$,
implying, by Lemma \foreignlanguage{british}{\hyperlink{Lemma 3}{3.4}},
that $H\mathcal{S}(U)$ is a continuum and thus (Lemma \foreignlanguage{british}{\hyperlink{Lemma 4}{4}})
that $H\mathcal{S}$ is locally a continuum, the difficulty of following
this approach lies in guarantying that $U$ is locally connected at
every $x\in\mbox{bd\,}U$. We overcome this difficulty taking a more
direct path: we show, through Peano Continuum Theory, that $H\mathcal{S}$
is locally a Peano continuum.
\selectlanguage{british}%
\begin{lem}
\hypertarget{Lemma 5}{}\foreignlanguage{english}{Generalized Peano
continua are locally Peano continua.}\end{lem}
\selectlanguage{english}%
\begin{proof}
We first show that 

(1) \emph{Peano continua are locally Peano continua}: let $X$ be
a Peano continuum. Given $x\in X$ and $U\in\mathcal{N}_{x}$, take
$\epsilon>0$ such that $B(x,\epsilon)\subset U$. $X$ is the union
of finitely many Peano continua $X_{i}$ with $\mbox{diam\,}X_{i}<\epsilon$
\cite[p.124, 8.10]{NA}. Let $V$ be the union of the $X_{i}$'s that
contain $x$, say $V=\cup_{i=1}^{^{n}}X_{i}$. $V$ is compact and
connected (i.e. a continuum) and $U\supset V\in\mathcal{N}_{x}$,
since $V^{c}$ is contained in a compact not containing $x$.
\selectlanguage{british}%
\begin{claim*}
$V$ \emph{is locally connected}: given $z\in V$ and a neighbourhood
$W$ of $z$ in $V$, let $\beta$ be the set of the $1\leq i\leq n$
such that $z\in X_{i}$. For each $i\in\beta$ take a connected neighbourhood
$C_{i}$ of $z$ in $X_{i}$ contained in $W$. Since the $X_{i}$'s
are compact and finite in number, it is easily seen that $C:=\cup_{i\in\beta}C_{i}\subset W$
is a connected neighbourhood of $z$ in $V$. Therefore, besides compact
and connected, $V$ is locally connected and thus a Peano continuum,
hence (1) is proved.
\end{claim*}
(2) \emph{Generalized Peano continua are locally Peano continua}:
let $X$ be a generalized Peano continuum. Since (2) coincides with
(1) if $X$ is compact assume it is not. Let $X^{\propto}=X\sqcup\{\infty\}$
be a 1-point compactification of $X$ (on $X$, the metric $d'$ of
$X^{\propto}$ is equivalent to the original metric $d$ of $X$).
Then $X^{\propto}$ is compact, connected and locally connected at
every point $x\in X=X^{\propto}\setminus\{\infty\}$ and thus also
at $\infty$ \cite[p.78, 5.13]{NA}. Hence $X^{\propto}$ is a Peano
continuum, therefore, by (1), $X^{\propto}$ is locally a Peano continuum
and so is its open subset $X$.\end{proof}
\selectlanguage{british}%
\begin{cor}
\hypertarget{Corollary 3}{}$H\mathcal{S}$ is
\begin{enumerate}
\item locally a Peano continuum,
\item the union of countably many disjoint, clopen, generalized Peano continua.
\end{enumerate}
\end{cor}
\begin{proof}
By Lemma \hyperlink{Lemma 4}{4}, $H\mathcal{S}$ is $d_{H}$ locally
connected, thus each component is open. As $H\mathcal{S}$ is also
$d_{H}$ locally compact, each component is actually a generalized
Peano continuum. Hence, (1) follows from Lemma \hyperlink{Lemma 5}{5}.
It remains to show that $H\mathcal{S}$ has countably many components.
$H\mathcal{S}$ is $d_{H}$ separable, since $\mbox{CMin}\supset H\mathcal{S}$
is (see \cite[p.258, Lemma 10]{TE}). The conclusion follows, since
the components of $H\mathcal{S}$ are open. 
\end{proof}
\selectlanguage{english}%
The next result is valid\foreignlanguage{british}{ for $C^{0}$ flows
on locally compact metric spaces.}
\selectlanguage{british}%
\begin{lem}
\hypertarget{Lemma 6}{}\foreignlanguage{english}{Let \emph{$\varLambda_{n\geq0}\in\mbox{CMin}$$.$}
Suppose that for any $\epsilon>0$, there is an $m\geq1$ such that
$n>m\implies\varLambda_{n}\subset B(\varLambda_{0},\epsilon)$. Then,
$\varLambda_{n}\overset{d_{H}}{\longrightarrow}\varLambda_{0}$. }\end{lem}
\selectlanguage{english}%
\begin{proof}
see \cite[p.255, Lemma 4]{TE}.
\end{proof}
From Lemma \foreignlanguage{british}{\hyperlink{Lemma 6}{6}} we easily
deduce
\selectlanguage{british}%
\begin{lem}
\hypertarget{Lemma 7}{}\foreignlanguage{english}{Given any neighbourhood
$\mathfrak{N}$ of $\varLambda\in H\mathcal{S}$ in $[H\mathcal{S},\, d_{H}]$,
for every sufficiently small neighbourhood $U$ of $\varLambda$ in
$M$, $H\mathcal{S}(U)$ is a neighbourhood of $\varLambda$ in $H\mathcal{S}$,
contained in $\mathfrak{N}$. If $U$ is open in $M$, then $H\mathcal{S}(U)$
is $d_{H}$ open in $H\mathcal{S}$.}\end{lem}
\selectlanguage{english}%
\begin{proof}
see \cite[p.255, Lemma 6 and its proof]{TE}.
\end{proof}
\selectlanguage{british}%
\hypertarget{Section 4}{}

\section{Global and local structure of $H\mathcal{S}$ and the dynamics around
it}

We are now in possession of all the tools needed to prove the main
results of this paper, both in the global and local settings (Sections
\hyperlink{Section 4.1.1}{4.1.1} and \hyperlink{Section 4.1.2}{4.1.2}).
If the flow is non-wandering, the inherent dynamical constraints make
these characterizations assume a particularly elegant form (Section
\hyperlink{Section 4.2}{4.2}). Section \hyperlink{Section 4.1.3}{4.1.3}
calls attention to a crucial difference between the local topologies
of $H\mathcal{S}\subset\mbox{CMin}$ and $H\mathcal{S}^{*}\subset M$
(this disparity vanishes in dimensions 1 and 2).
\begin{rem}
All results in this section hold for arbitrary continuous flows on
connected manifolds (possibly noncompact or with boundary, see Remark
\hyperlink{Remark 1.c}{1.c}).
\end{rem}
\hypertarget{Section 4.1}{}

\subsection{Arbitrary flows}

\hypertarget{Section 4.1.1}{}

\subsubsection{$H\mathcal{S}$ globally}
\begin{thm}
\hypertarget{Theorem 1}{}(Global Structure of $H\mathcal{S}$) Let
$\theta$ be a $C^{0}$ flow on a generalized Peano continuum $M$.
Endowed with the Hausdorff metric, $H\mathcal{S}$ is the union of
countably many disjoint, clopen, generalized Peano continua $X_{i}$
(each admitting a complete geodesic metric). Moreover, there are disjoint,
connected, open invariant sets $A_{i}\subset M$ such that:
\begin{enumerate}
\item $X_{i}^{*}\subset A_{i}$
\item $A_{i}=\{x\in M:\,\emptyset\neq\omega(x)\in X_{i}\subset H\mathcal{S}\}$
\item for any $x\in A_{i}$, $x\not\in\omega(x)$ implies \emph{$\alpha(x)\subset\mbox{bd}\, A_{i}$
}(possibly $\alpha(x)=\emptyset$).
\end{enumerate}

(see Fig. \hyperlink{Fig 1.4}{1.4})

\end{thm}
\begin{rem}
\textit{\emph{(a) if $M$ is compact and the flow is minimal, then
the unique $X_{i}$ is $\{M\}$ and $A_{i}=M$;}}

(b) if $H\mathcal{S}=\emptyset$, then the collection $\{X_{i}\}$
is empty;

(c) if $X_{i}$ contains a unique $\varLambda\neq M$, then $\varLambda$
is a proper attractor and $A_{i}$ its basin;

(d) otherwise, $X_{i}$ is a non-degenerate generalized Peano continuum
and thus contains $\mathfrak{c}$ hyper-stable compact minimals ($X_{i}$
contains nontrivial arcs of $H\mathcal{S}$);

(e) although $X_{i}^{*}$ may be noncompact, it roughly acts as an
attracting set in the flow, with basin $A_{i}$ (see Fig. \hyperlink{Fig 1.4}{1.4});

(f) if $A_{i}=M$ then, for every $x\in M$, either $x\in\omega(x)\in H\mathcal{S}$
or $\alpha(x)=\emptyset$ (and in the later case, $M$ is noncompact).\end{rem}
\begin{cor}
$H\mathcal{S}$ is countable iff it is $d_{H}$ discrete iff every
$\varLambda\in H\mathcal{S}$ is an attractor.
\end{cor}

\begin{proof}
(\emph{of Theorem }\hyperlink{Theorem 1}{1}) Let $\{X_{i}\}$ be
the components of $H\mathcal{S}$. The first conclusion of Theorem
\hyperlink{Theorem 1}{1} is (2) of Corollary \hyperlink{Corollary 3}{3},
together with Remark \hyperlink{Remark 1.c}{1.c}.

Let $A_{i}$ be defined as in (2). Then $\varLambda\in X_{i}\implies\omega(x)=\varLambda$,
for every $x\in\varLambda$, thus $\varLambda^{*}\subset A_{i}$,
hence (1) $X_{i}^{*}\subset A_{i}$. As the $X_{i}$'s are disjoint,
so are the $A_{i}$'s. Since points in the same orbit have the same
$\omega$-limit, each $A_{i}$ is invariant. It is also open: suppose
that $x\in A_{i}$ i.e. $\omega(x)\in X_{i}$. Take a neighbourhood
$U$ of $\omega(x)\in H\mathcal{S}$ in $M$ as in Lemma \hyperlink{Lemma 1}{1},
sufficiently small so that $H\mathcal{S}(U)\subset X_{i}$ (Lemma
\hyperlink{Lemma 7}{7}, using the fact that $X_{i}$ is open in $H\mathcal{S}$).
By continuity of the flow, taking a sufficiently small $B\in\mathcal{N}_{x}$,
$\mathcal{O}^{+}(y)\cap\mbox{int\,}U\neq\emptyset$, for every $y\in B$,
thus implying (Lemma \hyperlink{Lemma 3}{3.1}) $\omega(y)\in H\mathcal{S}(U)\subset X_{i}$
and thus $y\in A_{i}$. Therefore $A_{i}$ is open in $M$. $A_{i}$
is connected since $X_{i}^{*}$ is connected and for every $x\in A_{i}$,
$\mathcal{O}(x)\cup\omega(x)\subset A_{i}$ is connected and $\omega(x)\subset X_{i}^{*}$
($X_{i}^{*}$ is connected: this is trivial noting that any nontrivial
partition of $X_{i}^{*}$ into closed sets entails a nontrivial partition
of $X_{i}$ into closed sets, as each $\varLambda\in X_{i}$ is a
minimal, thus connected).

It remains to prove (3). Suppose $x\in A_{i}$ and $x\not\in\omega(x)$.
As $A_{i}$ is open and invariant, $\alpha(x)\subset\overline{A_{i}}=A_{i}\sqcup\mbox{bd\,}A_{i}$.
Reasoning by contradiction, suppose that $y\in\alpha(x)\cap A_{i}\neq\emptyset$.
Then, $\omega(y)\subset\alpha(x)$ and $\omega(y)\in X_{i}\subset H\mathcal{S}$.
Since $x\not\in\omega(y)$ (otherwise $x\in\omega(x)$, as $\omega(y)$
is minimal), $\omega(y)$ is unstable (Lemma \hyperlink{Lemma 2}{2.3}),
contradiction. Therefore $\alpha(x)\subset\mbox{bd\,}A_{i}$. If $M$
is noncompact it is obviously possible that $\alpha(x)=\emptyset$.
\end{proof}
\hypertarget{Section 4.1.2}{}

\subsubsection{$H\mathcal{S}$ locally}
\begin{thm}
\hypertarget{Theorem 2}{}(Local behaviour) Let $\varLambda$ be a
compact minimal set of a $C^{0}$ flow on a generalized Peano continuum
$M$. Then, either
\begin{enumerate}
\item \emph{$\varLambda\in\mbox{cl}_{H}\mathcal{U}$}, or
\item $\varLambda$ is an attractor, or
\item there are arbitrarily small compact, connected, (+)invariant neighbourhoods
$U$ of $\varLambda$ in $M$ such that:

\begin{enumerate}
\item the (compact) minimal sets contained in $U$ are all Lyapunov hyper-stable
and their set $H\mathcal{S}(U)$ is a non-degenerate Peano continuum;
\item for each $x\in U$, $\omega(x)\in H\mathcal{S}(U)$ and if $x\not\in\omega(x)$,
then its negative orbit leaves $U$ (and thus never returns again).
\end{enumerate}
\end{enumerate}

(see Fig. \hyperlink{Fig 1.3}{1.3})

\end{thm}
\begin{rem}
$\bullet$~(1) holds iff $\varLambda\not\in H\mathcal{S}$ i.e. if
there is a sequence $\mathcal{U}\ni\varLambda_{n}\overset{d_{H}}{\longrightarrow}\varLambda$
(Section \hyperlink{Section 2.3}{2.3}). If $\varLambda$ is unstable,
then this trivially holds since $\varLambda_{n}:=\varLambda$ is such
a sequence; 

$\bullet$~if $\varLambda=M$ (i.e. $M$ is compact and the flow
is minimal), then (2) trivially holds;

$\bullet$~(3.a) implies that, endowed with the Hausdorff metric,
$H\mathcal{S}(U)$\emph{ is compact, nontrivially arcwise connected
($\implies|H\mathcal{S}(U)|=\mathfrak{c}),$ locally arcwise connected
and admits a complete geodesic metric }(see Remark \hyperlink{Remark 1.c}{1.c});

$\bullet$~the final conclusion in (3.b) can be written as: 
\[
\exists\, T\leq0:\, x^{(-\infty,T)}\subset U^{c}\mbox{ and }\; x^{[T,+\infty)}\subset U
\]
(see the proof of Lemma \hyperlink{Lemma 3}{3}).\end{rem}
\begin{defn}
\hypertarget{Definition 8}{}$\varLambda\in\mbox{CMin}$ is \emph{isolated
(from minimals) }if there is an $U\in\mathcal{N}_{\varLambda}$ containing
no compact minimal set other than $\varLambda$. This is equivalent
to $\varLambda$ being $d_{H}$ isolated in $\mbox{CMin}$ (by the
$d_{H}$ metric definition and Lemma \hyperlink{Lemma 6}{6}).\end{defn}
\begin{cor}
\hypertarget{Corollary 6}{}If $\varLambda\in\mathcal{S}$ is isolated,
then it is an attractor.\end{cor}
\begin{proof}
If $\varLambda\in\mathcal{S}$ is isolated, then neither (1) nor (3)
of Theorem \hyperlink{Theorem 2}{2} can hold, since both imply the
existence of compact minimals $\varGamma\neq\varLambda$ contained
in every neighbourhood of $\varLambda$.
\end{proof}

\begin{proof}
(\emph{of Theorem }\hyperlink{Theorem 2}{2}) The three conditions
are mutually exclusive.

(A) if $\varLambda\not\in H\mathcal{S}$, then (1) holds (Section
\hyperlink{Section 2.3}{2.3}).

(B) suppose $\varLambda\in H\mathcal{S}$. Let $X_{i}$ be the component
of $H\mathcal{S}$ to which $\varLambda$ belongs (see the proof of
Theorem \hyperlink{Theorem 2}{2}). We distinguish two cases:

(B.1) if $X_{i}=\{\varLambda\}$, then $\varLambda$ is stable and
$A_{i}\in\mathcal{N}_{\varLambda}$ (Theorem \hyperlink{Theorem 1}{1})
is its region of attraction, thus $\varLambda$ is an attractor and
$A_{i}$ its basin. Hence (2) holds.

(B.2) otherwise, let $\mathfrak{N}$ be a Peano continuum neighbourhood
of $\varLambda$ in $H\mathcal{S}$, contained in $X_{i}$ ($X_{i}$
is open in $H\mathcal{S}$ (Theorem \hyperlink{Theorem 1}{1}) and
$H\mathcal{S}$ is locally a Peano continuum (Corollary \hyperlink{Corollary 3}{3})).
As $X_{i}$ is non-degenerate (i.e. $|X_{i}|>1$), so is $\mathfrak{N}$.
Take a sufficiently small compact, connected, (+)invariant neighbourhood
$V$ of $\varLambda$ in $M$ such that $\mbox{CMin}(V)\subset\mathfrak{N}$
(Lemmas \hyperlink{Lemma 1}{1}\foreignlanguage{english}{ and }\hyperlink{Lemma 7}{7}).
Observe that $\mathfrak{N}^{*}\subset M$ is (a) compact: $\mathfrak{N}$
is $d_{H}$ compact, hence given sequences $x_{n}\in\varLambda_{n}\in\mathfrak{N}$,
there is a convergent subsequence $\varLambda_{n'}\overset{d_{H}}{\longrightarrow}\varLambda\in\mathfrak{N}$.
We may assume that all $\varLambda_{n'}$'s are contained in some
compact neighbourhood $W$ of $\varLambda$ in $M$ ($\varLambda$
is compact and $M$ is locally compact). The conclusion follows since
$x_{n'}$ has a convergent subsequence $x_{n''}\rightarrow x\in W$
and necessarily $x\in\varLambda$, as $x_{n''}\in\varLambda_{n''}\overset{d_{H}}{\longrightarrow}\varLambda$;
(b) connected (as $X_{i}^{*}$ in the proof of Theorem \hyperlink{Theorem 1}{1})
and (c) invariant (union of minimal sets). Then $U:=V\cup\mathfrak{N}^{*}$
is a compact, connected, (+)invariant neighbourhood of $\varLambda$
in $M$ and $\mbox{CMin}(U)=\mathfrak{N}\subset H\mathcal{S}$. Thus
(3.a) is proved; (3.b) follows from Lemma \hyperlink{Lemma 3}{3.1}
and \hyperlink{Lemma 3}{3.3}, since $\mbox{CMin}(U)\subset H\mathcal{S}$.
\end{proof}
\hypertarget{Section 4.1.3}{}

\subsubsection{Topology of $H\mathcal{S}^{*}\subset M$}

It is easily seen that, as $H\mathcal{S}$, $H\mathcal{S}^{*}\subset M$
is locally compact and that $\{X_{i}^{*}\}$ are its (clopen) components
($\{X_{i}\}$ being the components of $H\mathcal{S}$). But while
$H\mathcal{S}$ is locally connected (in relation to the $d_{H}$
metric), the corresponding set of points $H\mathcal{S}^{*}$ needs
not to be a locally connected subset of $M$. The local topology of
the minimal sets $\varLambda\in H\mathcal{S}$ plays a determinant
role here. Actually, $H\mathcal{S}^{*}$ may be nowhere locally connected,
even if the flow is smooth. Handel's by product example \cite[p.166]{HA}
can be transferred to $\mathbb{S}^{2}$, to yield an orientation preserving
$C^{\infty}$ diffeomorphism $f:\mathbb{S}^{2}\circlearrowleft$ with
only three minimal sets, two repelling fixed points (the north and
south poles $\pm p$) and an attracting pseudo-circle $P$, with basin
$\mathbb{S}^{2}\setminus\{\pm p\}$. The pseudo-circle is nowhere
locally connected. Taking the suspension of $f$, we get a $C^{\infty}$
flow $v^{t}$ on $\mathbb{S}_{f}^{2}\simeq\mathbb{S}^{2}\times\mathbb{S}^{1}$,%
\footnote{We may actually identify these two manifolds and suppose $v^{t}$
defined on $\mathbb{S}^{2}\times\mathbb{S}^{1}$, as homeomorphic
3-manifolds are $C^{\infty}$ diffeomorphic.%
} with a nowhere locally connected attracting minimal set $\varLambda=P_{f}$
(locally, $\varLambda$ is homeomorphic to the Cartesian product of
$\mathbb{D}^{1}$ and some open subset of $P$, hence it is nowhere
locally connected). In this flow, $H\mathcal{S}=\{\varLambda\}\simeq\{0\}$
and $H\mathcal{S}^{*}=\varLambda$, therefore $H\mathcal{S}^{*}$
is nowhere locally connected. Examples of $C^{\infty}$ flows in higher
dimensions, with $H\mathcal{S}^{*}$ nowhere locally connected, are
generated by the vector fields
\[
(x,y)\longmapsto\big(v(x),0\big)\mbox{ \,\,\,\ on\,\,\,\ }(\mathbb{S}^{2}\times\mathbb{S}^{1})\times\mathbb{S}^{k},\mbox{ \,\,\,}k\geq1
\]
where $v$ is the original vector field on $\mathbb{S}^{2}\times\mathbb{S}^{1}$.
Then,
\[
H\mathcal{S}=\big\{ P_{f}\times\{y\}:\, y\in\mathbb{S}^{k}\big\}\simeq\mathbb{S}^{k}\mbox{ \,\,\,\ and \,\,\,\ensuremath{H\mathcal{S}^{*}=P_{f}\times\mathbb{S}^{k}}}
\]
is locally homeomorphic to the Cartesian product of $\mathbb{D}^{k+1}$
and an open subset of $P$, hence nowhere locally connected.

However, for $C^{0}$ flows on arbitrary 2-manifolds (possibly nonorientable,
noncompact, with boundary), $H\mathcal{S}^{*}$ is \emph{always }a
locally connected subset of $M$. Actually, each (clopen) component
$X_{i}^{*}$ of $H\mathcal{S}^{*}$ either
\begin{enumerate}
\item contains more than 2 equilibria, in which case it consists entirely
of equilibria, hence $X_{i}^{*}\simeq X_{i}$, implying that $X_{i}^{*}$
is locally connected (see Theorem \hyperlink{Theorem 1}{1}), or
\item contains no more than 2 equilibria, in which case $X_{i}^{*}$ is
a (connected) $k$-manifold ($0\leq k\leq2)$.
\end{enumerate}
The key fact to establish (2) is the following result of Athanassopoulos
and Strantzalos \cite{AT}:
\begin{thm*}
A Lyapunov stable compact minimal set of a $C^{0}$ flow on an arbitrary
2-manifold is either an equilibrium orbit, or a periodic orbit, or
a torus.
\end{thm*}
The reader is invited to look up for the 13 possible manifolds (up
to homeomorphism) that might occur as $X_{i}^{*}$ in case (2).

\hypertarget{Section 4.2}{}

\subsection{Non-wandering flows}

Note that, in virtue of Poincaré Recurrence Theorem, all results below
are valid, not only, for conservative flows on finite volume manifolds,
but also for any flows topologically equivalent to these ones (being
non-wandering is a topological equivalence invariant). This includes
conjugations via homeomorphisms and time reparametrizations. 
\begin{rem}
The results in this section are, in general, false if the flow is
\emph{not} non-wandering (even when $M$ is compact), the north-south
flow on $\mathbb{S}^{n}$ being a trivial counter-example to all of
them. \end{rem}
\begin{thm}
\hypertarget{Theorem 3}{}(Global structure of $H\mathcal{S}$) Given
a $C^{0}$ non-wandering flow on a generalized Peano continuum $M$,
either
\begin{enumerate}
\item $H\mathcal{S}=\emptyset$ i.e. the Lyapunov unstable compact minimal
sets are $d_{H}$ dense in \emph{$\mathcal{\mbox{CMin}},$} or
\item $H\mathcal{S}=\{M\}$ i.e. $M$ is compact and the flow is minimal,
or
\item $H\mathcal{S}$ is the union of $1\leq\beta\leq\aleph_{0}=|\mathbb{N}|$
disjoint, clopen, non-degenerate generalized Peano continua $X_{i\in\beta}$,
each $X_{i}^{*}\subset M$ being a (nonvoid) connected, open invariant
set.
\end{enumerate}
\end{thm}
\begin{proof}
The three conditions are mutually exclusive. Assume (1) and (2) fail.
Then, by Theorem \hyperlink{Theorem 1}{1}, $H\mathcal{S}$ is the
union of $1\leq\beta\leq\aleph_{0}$ clopen generalized Peano continua
$X_{i\in\beta}$ and these are non-degenerate since $|X_{i}|=1$ implies
the unique $\varLambda\in X_{i}$ is a proper attractor (as (2) fails,
$\varLambda\subsetneq M$), which is impossible in the non-wandering
context (as minimal attractors are both stable and isolated, see \hyperlink{I}{(1)}
in Section \hyperlink{Section 2.2}{2.2}). Let $A_{i}$ be as in Theorem
\hyperlink{Theorem 1}{1}. Then, $x\in A_{i}$ implies $\emptyset\neq\omega(x)\in X_{i}\subset H\mathcal{S}$
and $x\in\omega(x)$, as $\omega(x)$ is stable and the flow is non-wandering
(Lemma \hyperlink{Lemma 2}{2.4}). Therefore, $X_{i}^{*}=A_{i}$ and
$X_{i}^{*}$ is as claimed.
\end{proof}
The following result shows that any cardinal limitation on the number
of compact minimal sets implies either the minimality of the flow
or the $d_{H}$ denseness in $\mbox{CMin }$of the unstable compact
minimal sets.
\begin{cor}
\hypertarget{Corollary 6}{}If the flow is non-wandering and \emph{$|\mbox{CMin}|<\mathfrak{c}$},
then either 
\begin{enumerate}
\item $H\mathcal{S}=\emptyset$ ~i.e.~ \emph{$\mbox{cl}_{H}\mathcal{U}=\mbox{CMin}$,
}or
\item $M$ is compact and the flow is minimal. 
\end{enumerate}
Actually, in case (1), the isolated (and thus Lyapunov unstable) compact
minimal sets are $d_{H}$ dense in \emph{$\mbox{CMin}$.}\end{cor}
\begin{proof}
Case (3) of Theorem \hyperlink{Theorem 3}{3} implies that $|\mbox{CMin}|=\mathfrak{c}$,
thus $|\mbox{CMin}|<\mathfrak{c}$, implies that either (1) or (2)
holds. The last sentence is an immediate consequence of Theorem \hyperlink{Theorem 2}{2}
in \cite[p.248]{TE}: $|\mbox{CMin}|<\mathfrak{c}$ implies $\mathfrak{M}_{10}=\emptyset$,
thus $\mathfrak{M}_{1-6}$, the set of isolated compact minimals,
is $d_{H}$ dense in $\mbox{CMin}$.\end{proof}
\begin{rem}
Note that, in general, despite the abundance of recurrent points,
compact minimal sets may be extremely scarce in non-wandering flows
(e.g. the only $\varLambda\in\mbox{CMin}$ may be an equilibrium orbit).
Corollary \hyperlink{Corollary 6}{6} shows that this phenomenon may
only occur if isolated (unstable) compact minimal sets are $d_{H}$
dense in $\mbox{CMin}$.\end{rem}
\begin{thm}
\hypertarget{Theorem 4}{}(Local behaviour) Let $\varLambda$ be a
compact minimal set of a $C^{0}$ non-wandering flow on a generalized
Peano continuum $M$. Then, either
\begin{enumerate}
\item \emph{$\varLambda\in\mbox{cl}_{H}\mathcal{U}$}, or
\item $\varLambda=M$ i.e. $M$ is compact and the flow is minimal, or
\item there are arbitrarily small compact, connected, invariant neighbourhoods
$U$ of $\varLambda$ such that:

\begin{enumerate}
\item $U$ is the union of $\mathfrak{c}$ hyper-stable compact minimal
sets i.e. $U=H\mathcal{S}(U)^{*}$;
\item $H\mathcal{S}(U)$ is a (non-degenerate) Peano continuum.
\end{enumerate}
\end{enumerate}
\end{thm}
(see Fig. \hyperlink{Fig 1.2}{1.2}).
\begin{proof}
Theorem \hyperlink{Theorem 4}{4} is just Theorem \hyperlink{Theorem 2}{2}
in the non-wandering setting: (1) coincides in both theorems. If the
flow is non-wandering, then (2) of Theorem \hyperlink{Theorem 2}{2}
can hold iff $\varLambda=M$ since there are no proper attractors.
This is (2) of Theorem \hyperlink{Theorem 4}{4}. Finally, let $U\in\mathcal{N}_{\varLambda}$
be as in (3) of Theorem \hyperlink{Theorem 2}{2}. By (3.b) of Theorem
\hyperlink{Theorem 2}{2},
\[
x\in U\implies\omega(x)\in H\mathcal{S}(U)
\]
As the flow is non-wandering, this implies $x\in\omega(x)\subset U$
(Lemma \hyperlink{Lemma 2}{2.4}). Therefore, $U=H\mathcal{S}(U)^{*}$
is invariant and (3) holds.
\end{proof}
\hypertarget{Section 5}{}

\section{Global absence of Lyapunov instability}

\subsection{Dichotomy}

It is well know that transitive dynamical behaviour precludes the
existence of Lyapunov stable compact minimal sets $\varLambda\subsetneq M$.
At the other extreme of the spectrum are flows without Lyapunov unstable
compact minimals. \emph{If the phase space $M$ is compact}, this
imposes extremely strong dynamical constraints on the flow, the following
dichotomy holding:
\begin{thm}
\hypertarget{Theorem 5}{}(minimality or fragmentation into $\mathfrak{c}$
stable minimals) A $C^{0}$ flow $\theta$ without Lyapunov unstable
compact minimals on a Peano continuum $M$ is non-wandering. Actually,
every point is recurrent and the flow is either
\begin{enumerate}
\item minimal, or
\item partitions $M$ into $\mathfrak{c}$ Lyapunov hyper-stable compact
minimal sets, forming a non-degenerate Peano continuum (in the Hausdorff
metric). Every nonvoid open set $A\subset M$ intersects $\mathfrak{c}$
minimal sets.
\end{enumerate}
\end{thm}
Thus, in a certain sense, Lyapunov unstable compact minimal sets are
a vital ingredient of dynamical complexity and diversity of flows
on compact phase spaces.
\begin{proof}
As $\mathcal{U}=\emptyset$, $\mbox{CMin}=H\mathcal{S}$. Since $M$
is compact, for each $x\in M$, $\alpha(x)\neq\emptyset$ is compact
and thus contains a compact minimal set $\varGamma$. Necessarily,
$x\in\varGamma=\alpha(x)$, otherwise $\varGamma$ is unstable (Lemma
\hyperlink{Lemma 2}{2.3}). Thus, every $x$ belongs to some $\varGamma\in H\mathcal{S}$
and the flow is non-wandering, with all points recurrent. Since $H\mathcal{S}^{*}=M$,
if the flow is not minimal, then $|H\mathcal{S}|>1$, and by Lemma
\hyperlink{Lemma 3}{3.4} applied to $U:=M$, $H\mathcal{S}$ is $d_{H}$
compact and connected, hence by Theorem \hyperlink{Theorem 3}{3.3},
a non-degenerate Peano continuum (thus $|H\mathcal{S}|=\mathfrak{c}$).
It remains to prove the last sentence in (2). Let $x\in M$ and $\epsilon>0$.
We show that $B(x,\epsilon)$ intersects $\mathfrak{c}$ minimals:
$x$ belongs to some $\varLambda\in H\mathcal{S}$; as $H$$\mathcal{S}$
is a non-degenerate Peano continuum, there are $\mathfrak{c}$ distinct
minimals $\varGamma_{i\in\mathbb{R}}\in H\mathcal{S}$ such that $d_{H}(\varGamma_{i},\varLambda)<\epsilon$,
thus implying $\varGamma_{i}\cap B(x,\epsilon)\neq\emptyset$.\end{proof}
\begin{rem}
The last conclusion in (2) means that the ``local density'' of the
hyper-stable compact minimal sets is actually $\mathfrak{c}$ all
over $M$. 
\end{rem}

\begin{rem}
The same conclusion holds if the flow has no (-)unstable compact minimal
sets, as the time reversed flow $\phi(t,x)=\theta(-t,x)$ contains
no unstable compact minimal set and thus Theorem \hyperlink{Theorem 5}{5}
is valid for $\phi$, hence also for $\theta,$ as its conclusions
are preserved under time-reversal.
\end{rem}

\begin{rem}
This result, valid for arbitrary flows on compact (connected) manifolds,
is, in general, false when $M$ is noncompact (trivial counter-examples
include the flow on $\mathbb{R}^{n}$ generated by the vector field
$v(z)=-z$ and the parallel flows $\partial/\partial x_{i}$. If $M$
is noncompact, the flow may contain no minimal sets, compact or not,
see Inaba \cite{IN}, Beniere, Meigniez \cite{BE}). Actually, compactness
of $M$ plays the key role in the above proof. If $M$ is noncompact
(i.e. a noncompact generalized Peano continuum, e.g. a noncompact
manifold), then limit sets may be empty or noncompact and thus need
not contain compact minimal sets. This actually implies that, in the
noncompact setting, absence of Lyapunov unstable compact minimals
has no analogue constraining effect on the dynamical diversity and
complexity of the flow. In particular, $H\mathcal{S}$ may be empty
or discrete (i.e. every $\varLambda\in H\mathcal{S}$ may be an attractor).
Without entering into details, these flows may exhibit ``quite freely''
(on $M$ or on invariant noncompact submanifolds), an abundance of
dynamical phenomena e.g. non-minimal ergodic behaviour, absence of
minimal sets (compact and noncompact) etc.
\end{rem}
\hypertarget{Section 5.2}{}

\subsection{Fragmentation into $\mathfrak{c}$ hyper-stable minimal submanifolds}

In virtue of Theorem \hyperlink{Theorem 5}{5}, a non-minimal $C^{0}$
flow on a (connected) closed manifold $M$, displaying no unstable
minimal sets, partitions $M$ into $\mathfrak{c}$ hyper-stable minimal
sets $\{\varLambda_{i\in\mathbb{R}}\}$. If all these $\varLambda_{i}$
are submanifolds (necessarily closed and connected, of possibly non-fixed
codimension $\geq1$), then given the nature of the examples that
usually come to one's mind, it is tempting to ask if $[\mbox{CMin},d_{H}]=[\{\varLambda_{i\in\mathbb{R}}\},d_{H}]$\foreignlanguage{english}{
is itself a manifold (compact, connected, possibly with boundary).
For simplicity reasons, we will assume that the }flows has all orbits
periodic\foreignlanguage{english}{ and Lyapunov stable, hence $[\mbox{CMin},d_{H}]=[\mbox{Per},d_{H}]=H\mathcal{S}$.
In this particular case, the answer is positive in dimension 2 and
3 (Davermann \cite{D2,D4}, Davermann, Walsh \cite{D3}).}%
\footnote{Observe that the partitions (decompositions) of $M$ into closed submanifolds
we are considering are upper semicontinuous (usc) in the standard
sense \cite{B3,D1}, as required in \cite{D2,D3,D4}. This follows
immediately from Lemma \hyperlink{Lemma 1}{1}. They are in fact \emph{continuous}
in the standard sense \cite{B3}: $d(\varLambda_{n},\,\varLambda_{0})\rightarrow0$
implies $\varLambda_{n}\overset{d_{H}}{\longrightarrow}\varLambda_{0}$
(Lemmas \hyperlink{Lemma 1}{1} and \hyperlink{Lemma 6}{6}).%
}\foreignlanguage{english}{ }In dimension 2 this is actually straightforward,
using Gutierrez Smoothing Theorem \cite{GU} and Poincaré Index Theorem:
only the torus and the Klein bottle carry flows without equilibria.
For flows with all orbits periodic on $\mathbb{T}^{2},$ $[\mbox{Per},d_{H}]\simeq\mathbb{S}^{1}$
and on $\mathbb{K}^{2}$,\foreignlanguage{english}{ $[\mbox{Per},d_{H}]\simeq\mathbb{D}^{1}=[-1,1]$
(two $\mathbb{S}^{1}$-foliated Möbius bands glued along their boundaries$\implies$exactly
two 1-sided orbits).} The 3-dimensional case was established earlier
by Epstein \cite{E1}, assuming that the flows with all orbits periodic
is $C^{1}$ and $M$ orientable, with no a priori restrictions on
the Lyapunov's nature of the orbits (see Section \hyperlink{Section 2.5}{2.5}).
However, in the higher dimensions, the landscape changes radically:
there are $C^{\infty}$ periodic flows on closed $n$-manifolds (for
every $n\geq4$), for which $[\mbox{CMin},d_{H}]=[\mbox{Per},d_{H}]=H\mathcal{S}$
is a \emph{non}-manifold (see Example \hyperlink{Example 4}{4} below).

Cannon and Daverman \cite{CA} constructed remarkable examples of
periodic $C^{0}$ flows on $M=N\times\mathbb{S}^{1}$, $N$ any $C^{\infty}$
boundaryless $(n\geq3)$-manifold, on which every orbit is a wildly
embedded $\mathbb{S}^{1}$! By construction, these periodic flows
induce trivial circle bundles. As the fibres are wild in $M$, no
point of the bundle's base space $\varTheta$ has a neighbourhood
homeomorphic to $\mathbb{B}^{n}$, hence $H\mathcal{S}\simeq\varTheta$
is nowhere a manifold. These flows have nowhere an $n$-cell cross
section, and thus are nowhere topologically equivalent to $C^{1}$
flows, all this showing that $C^{0}$ dynamics harbours topological
phenomena unparalleled in the differentiable setting (even locally).
\begin{example}
\hypertarget{Example 4}{}The following construction provides examples
of $C^{\infty}$ periodic flows with all orbits Lyapunov stable, on
closed manifolds in all dimensions $n\geq4$, for which the space
of orbits $\mbox{Per}\simeq\varTheta$ is a \emph{non}-manifold. The
construction of the underlying tangentially orientable foliations,
which could hardly be simpler, was, essentially, kindly communicated
to us by Professor Robert Daverman \cite{D5}.

For $n\geq4$, let $f:\mathbb{S}^{n-1}\circlearrowleft$ act as the
orthogonal reflection on the north-south axis $[-p,p]$, $p=(0,0,\ldots,0,1)$
(this is the compactification of $\mathbb{R}^{n-1}\circlearrowleft:x\mapsto-x$).
The (semifree) $\mathbb{Z}/2\mathbb{Z}$ action determined by $f$
fixes $\pm p$ and the orbit space $\varTheta:=\mathbb{S}^{n-1}/f$
is homeomorphic to the topological suspension of $\mathbb{\mathbb{R}}P^{n-2}$.
As $n\geq4$, $\mathbb{R}P^{n-2}\not\hookrightarrow\mathbb{R}^{n-1}$
\cite{MA}, hence $\varTheta$ is a non-manifold (with singular points
$\pm p$). 

Let $F$ be the free $\mathbb{Z}/2\mathbb{Z}$ action on $\mathbb{S}^{1}\times\mathbb{S}^{n-1}$
which acts as the antipodal map on the 1st factor and as $f$ on the
2nd. Since the action is $C^{\infty},$ finite and free, the corresponding
orbit space $M=(\mathbb{S}^{1}\times\mathbb{S}^{n-1})/F$ is a connected,
$C^{\infty}$ closed $n$-manifold. The image of the circles $\mathbb{S}^{1}\times\{y\}$,
under the quotient map $h:\mathbb{S}^{1}\times\mathbb{S}^{n-1}\rightarrow M$,
are circles defining a tangentially orientable $C^{\infty}$ 1-foliation
of $M$, with all leaves Lyapunov stable. Now, starting with the periodic
vector field $(z_{1},z_{2})\mapsto(iz_{1},0)$ on $\mathbb{S}^{1}\times\mathbb{S}^{n-1}$,
we get, via the quotient map, a $C^{\infty}$ vector field on $M$,
tangent to the resulting foliation. Thus $H\mathcal{S}=\mbox{Per}$
is the space of leaves, which, by construction, is homeomorphic to
$\varTheta$, a non-manifold (see above). Also, trivially, the flow
is periodic with period $2\pi$ (identifying $\mathbb{S}^{1}$ with
$\mathbb{R}/2\pi\mathbb{\mathbb{Z}}$), the two orbits corresponding
to the image (under the quotient) of each circle $\mathbb{S}^{1}\times\{\pm p\}$
have minimal period $\pi$. All other orbits have minimal period $2\pi$.
\end{example}
\hypertarget{Section 6}{}

\section{Final remarks. Open questions}

Assuming the phase space $M$ of the flow to be a (generalized) Peano
continuum, the topological characterization of $H$$\mathcal{S}$
given by Theorem \hyperlink{Theorem 1}{1} is optimal: it is easily
seen that if a metric space $\mathfrak{M}$ is the union of countably
many disjoint, clopen, generalized Peano continua, then there is a
$C^{0}$ flow on a Peano continuum for which $H\mathcal{S}\simeq\mathfrak{M}$.
We sketch the proof in the case $\mathfrak{M}$ is noncompact and
has denumerably many components (the other cases are easier).

Let $\mathfrak{M}=\sqcup_{i\in\mathbb{N}}X_{i}$, where each $X_{i}$
is a nonvoid, clopen, generalized Peano continuum. For each $i\in\mathbb{N}$,
take a $C^{0}$ flow $\theta_{i}$ on $\mathbb{D}_{i}^{1}=[-1_{i},1_{i}]\simeq\mathbb{D}^{1}$,
with (exactly) three equilibria $-1_{i}$, $0_{i}$, $1_{i}$, $\{0_{i}\}$
a repeller. From each $X_{i}$ select a point $z_{i}$. Connect $z_{i}$
to $z_{i+1}$ pasting $-1_{i}$ to $z_{i}$ and $1_{i}$ to $z_{i+1}$
(the $\mathbb{D}_{i}^{1}$'s are disjoint, except that the pasting
induces the identification $1_{i}\equiv-1_{i+1}$). Define the $C^{0}$
flow $\theta$ on $\mathfrak{N}=\mathfrak{M}\cup\cup_{i\in\mathbb{N}}\mathbb{D}_{i}^{1}$,
which coincides with $\theta_{i}$ on $\mathbb{D}_{i}^{1}$ and has
each $x\in\mathfrak{M}$ has an equilibrium. $\mathfrak{N}$ is a
noncompact generalized Peano continuum and thus has an 1-point compactification
$\mathfrak{N}^{\propto}=\mathfrak{N}\sqcup\{0_{\infty}\}$, which
is a Peano continuum (see (2) in the proof of Lemma \hyperlink{Lemma 5}{5}).
The flow $\theta$ automatically extends to a $C^{0}$ flow on $\mathfrak{N}^{\propto}$,
with $0_{\infty}$ becoming an equilibrium. Now, let $\phi$ be a
$C^{0}$ flow on $\mathbb{D}^{1}$ with (exactly) two equilibria,
$-1$ and $1$, $\{-1\}$ an attractor. Paste $-1$ to $z_{1}$ and
$1$ to $0_{\infty}$. This defines a $C^{0}$ flow on the Peano continuum
\[
M:=\mathfrak{N}^{\propto}\cup\mathbb{D}^{1}
\]
 with $\mathcal{S}=H\mathcal{S}=\big\{\{x\}:\, x\in\mathfrak{M}\big\}\simeq\mathfrak{M}$
and $\mathcal{U}=\big\{\{0_{i}\}:\, i\in\mathbb{N}\sqcup\{\infty\}\big\}$.\medskip{}

A more difficult question is the following:
\begin{problem}
Assuming that $\mathfrak{M}$ (see above) is finite dimensional,%
\footnote{Every $n$-dimensional, separable metric space embeds in $\mathbb{R}^{2n+1}$,
hence in any $(2n+1)$-manifold (Menger-Nöbeling-Hurewicz Theorem,
see e.g. \cite[p.60]{HU}).%
}

when is $\mathfrak{M}$ homeomorphic to the $H\mathcal{S}$ set of
some flow on a manifold?
\end{problem}
We restrict our attention to the simpler problem:
\begin{problem}
\hypertarget{Hypothesis}{}If $K\subset\mathbb{S}^{n}$ is a Peano
continuum, under which conditions is there a flow on $\mathbb{S}^{n}$
such that:

(a) each point $x\in K$ is an equilibrium and

(b) $H\mathcal{S}=\big\{\{x\}:x\in K\big\}\simeq K$~?
\end{problem}
For $n=2$, such a flow exists iff $K^{c}$ has finitely many components,
and it can be made of class $C^{\infty}$ (if $\mathbb{S}^{2}$ is
replaced by any compact manifold, it is easily seen that this condition
remains necessary for the existence of such a flow, even of class
$C^{0}$, see \hyperlink{(B)}{(B)} below). Hence the answer is positive,
for example, if $K$ is homeomorphic to Wazewski's universal dendrite
(\cite[p.181]{NA},\cite[p.12]{CH}), but negative if it is homeomorphic
to Sierpinski's universal plane curve (``Sierpinski's carpet'',
\cite[p.9]{NA},\cite[p.31]{CH},\cite[p.160]{SA}). Our existence
proof relies heavily on Riemann Mapping Theorem.

\begin{figure}
\hypertarget{Fig 6.1}{}

\begin{centering}
\includegraphics[scale=0.85]{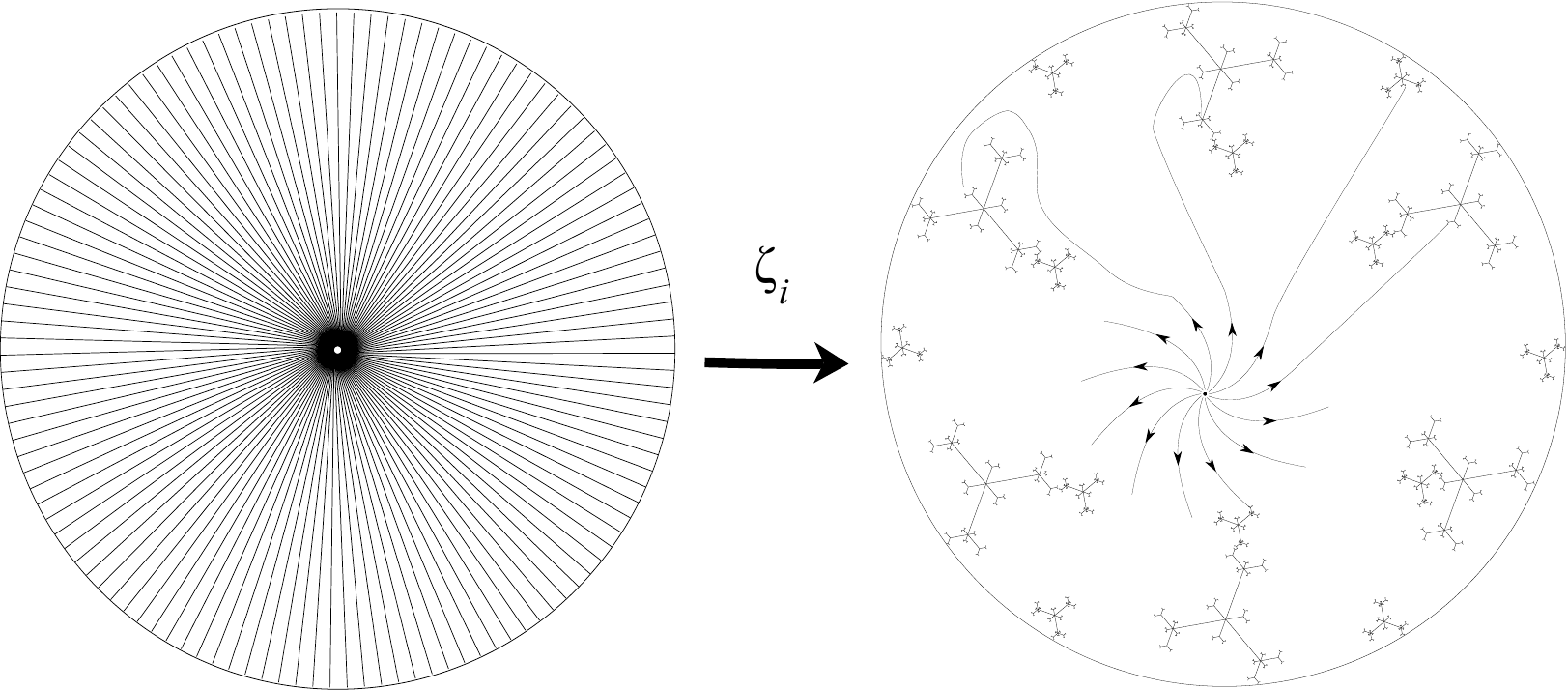}
\par\end{centering}

\caption{{\small{Riemann Mapping Theorem operating the miracle}}.}

\end{figure}

\textbf{Synopsis. }(\emph{Existence}) excluding trivialities, suppose
$K\subsetneq\mathbb{S}^{2}$ is a non-degenerate Peano continuum.
Through topology, each component $A_{i}$ of $K^{c}$ is simply connected,
hence there is a biholomorphism $\zeta_{i}:\mathbb{B}^{2}\rightarrow A_{i}$
(Riemann Mapping Theorem). We use this to put a $C^{\infty}$ vector
field $v_{i}$ on each $A_{i}$, with a repelling equilibrium $O_{i}$,
such that $A_{i}$ is its repulsion basin and $v_{i}$ extends to
the whole $\mathbb{S}^{2}$, letting $v_{i}=0$ on $A_{i}^{c}$. Define
$v=\sum v_{i}$ on $\mathbb{S}^{2}$. Clearly $v|_{A_{i}}=v_{i}$
and $\mbox{CMin}=\big\{\{x\}:\, x\in K\big\}\sqcup\big\{\{O_{i}\}\big\}_{i}$.
By topology again, $\mbox{bd\,}A_{i}$ is locally connected, hence
$\zeta_{i}$ extends continuously to $\zeta_{i}:\mathbb{D}^{2}\rightarrow\overline{A_{i}}$.
This ensures that, for each $p\in\mbox{bd\,}A_{i}\subset K$, $\{p\}$
is a stable equilibrium orbit, actually hyper-stable, since the unstable
minimals are the finitely many repellers $\{O_{i}\}$.
\begin{proof}
(A) \emph{Existence.} The cases $K=\emptyset$, $|K|=1$ and $K=\mathbb{S}^{2}$
are trivial. Assume that $K\subsetneq\mathbb{S}^{2}$ is non-degenerate
Peano continuum such that $K^{c}$ has finitely many components $\{A_{i}\}$.
Identify $\mathbb{R}^{2}$ with $\mathbb{C}$ and $\mathbb{S}^{2}$
with the Riemann sphere $\mathbb{C}\cup\{\infty\}$.

(1)\emph{ Claim. Each $A_{i}$ is biholomorphic to }$\mathbb{B}^{2}=\mbox{int\,}\mathbb{D}^{2}$:
let $\gamma$ be an $\mathbb{S}^{1}$ embedded in $A_{i}$. By Schoenflies
Theorem, we may reason as if $\gamma$ is the standard equator $\mathbb{S}^{1}\times\{0\}$.
Being connected and disjoint from $\gamma$, $K$ is contained in
one open hemisphere, say the north one. Then, the closed south hemisphere
is contained in $A_{i}$ (being a connected subset of $K^{c}$ containing
$\gamma\subset A_{i}$), hence $\gamma$ is contractible to a point
inside $A_{i}$. Therefore, as $A_{i}\neq\emptyset$ is open, simply
connected and $|A_{i}^{c}|>1$, there is a biholomorphic map $\zeta_{i}:\mathbb{B}^{2}\rightarrow A_{i}\subset\mathbb{S}^{2}$
(Riemann Mapping Theorem).

(2) as $K$ is Peano continuum, so is $\mbox{bd\,}A_{i}\subset K$
(\cite{KU}). This implies (\cite[p.18]{PO}) that $\zeta_{i}$ (uniquely)
extends to a $C^{0}$ map $\zeta_{i}:\mathbb{D}^{2}\rightarrow\overline{A_{i}}$.
It is easily seen that $\zeta_{i}$ maps $\mathbb{S}^{1}=\mbox{bd\,}\mathbb{D}^{2}$
onto $\mbox{bd\,}A_{i}=\overline{A_{i}}\setminus A_{i}$ (in general
not injectively).

(3) Take $\lambda\in C^{\infty}(\mathbb{R}^{2},[0,1])$ such that
$\lambda^{-1}(0)=(\mathbb{B}^{2})^{c}$. Let $\upsilon$ be the vector
field $\mathbb{R}^{2}\circlearrowleft:\, z\mapsto\lambda(z)z$. Transfer
$\upsilon|_{\mathbb{B}^{2}}$ to $A_{i}$ via $\zeta_{i}$, getting
$\upsilon_{i}=\zeta_{i_{*}}\upsilon|_{\mathbb{B}^{2}}\in\mathfrak{X}^{\infty}(A_{i})$.
By Kaplan Smoothing Theorem \cite[p.157]{KA}, there is $\mu_{i}\in C^{\infty}(\mathbb{S}^{2},[0,1])$
such that $\mu_{i}^{-1}(0)=A_{i}^{c}$ and 
\[
\begin{array}{lllc}
v_{i}:\mathbb{S}^{2} & \longrightarrow & \mathbb{R}^{3}\\
\quad\;\,\, z & \longmapsto & \mu_{i}\upsilon_{i}(z) & \mbox{ on \,}A_{i}\\
\quad\;\,\, z & \longmapsto & 0 & \mbox{ on \,}A_{i}^{c}
\end{array}
\]
defines a $C^{\infty}$ vector field on $\mathbb{S}^{2}$, whose restriction
to $A_{i}$ is topologically equivalent to $\upsilon|_{\mathbb{B}^{2}}$
via $\zeta_{i}$. Let $v=\sum v_{i}\in\mathfrak{X}^{\infty}(\mathbb{S}^{2})$.
Note that $v|_{A_{i}}=v_{i}.$ Its set of equilibria is $K\sqcup\{O_{i}\}_{i}$,
where $O_{i}=\zeta_{i}(0)$. The corresponding equilibrium orbits
are the only minimal sets of the flow $v^{t}$. Each $\{O_{i}\}$
is a repeller and $A_{i}$ its repulsion basin. For each $O_{i}\neq z\in A_{i}$,
$\alpha(z)=\{O_{i}\}$ and $\omega(z)=\{p\}$, for some $p\in\mbox{bd\,}A_{i}$.
Since $\zeta_{i}:\mathbb{S}^{1}\rightarrow\mbox{bd\,}A_{i}$ is onto,
every equilibrium $p\in\mbox{bd\,}A_{i}$ is the $\omega$-limit of
at least one $z\in A_{i}$ (Fig. \hyperlink{Fig 6.1}{6.1}).

(4) \emph{Claim.} $H\mathcal{S}=\big\{\{x\}:\, x\in K\big\}$. 

We show that each $y\in K$ has arbitrarily small (+)invariant neighbourhoods.
If $y\in\mbox{int\,}K$, this is obvious since $y$ has a neighbourhood
consisting of equilibria (see 3). Otherwise, given $\epsilon>0$,
we get a (+)invariant neighbourhood $D_{i}\subset B(y,\epsilon)$
of $y$ in $\overline{A_{i}}$, for each component $A_{i}$ of $K^{c}$
such that $y\in\mbox{bd\,}A_{i}$. Then, as the number of components
is finite, $\big(B(y,\epsilon)\cap K\big)\cup(\cup D_{i})\subset B(y,\epsilon)$
is a (+)invariant neighbourhood of $y$ in $\mathbb{S}^{2}$. 

Suppose $y\in\mbox{bd\,}A_{i}$. As $\zeta_{i}:\mathbb{D}^{2}\rightarrow\overline{A_{i}}$
is $C^{0}$, $\beta_{i}=\zeta_{i}^{-1}(y)\subset\mathbb{S}^{1}$ is
compact and $B_{i}=\zeta_{i}^{-1}\big(B(y,\epsilon)\cap\overline{A_{i}}\big)$
is an open neighbourhood of $\beta_{i}$ in $\mathbb{D}^{2}$. For
each $x\in\beta_{i}$, take a ``conic'' open, (+)invariant neighbourhood
$C_{x}$ of $x$ in $\mathbb{D}^{2}$, contained in $B_{i}$ ($\mathbb{D}^{2}$
is invariant under the flow $\upsilon^{t}$). Let $C_{i}$ be a finite
union of $C_{x}$'s covering $\beta_{i}$. Then, $\big(B(y,\epsilon)\cap K\big)\cup\big(\cup_{y\in\mbox{bd\,}A_{i}}\zeta_{i}(C_{i})\big)$
is a (+)invariant neighbourhood of $y$ in $\mathbb{S}^{2}$, contained
in $B(y,\epsilon)$. Therefore $\{y\}\in\mathcal{S}$. The only other
minimals are the finitely many repellers $\{O_{i}\}$, which are necessarily
away from $\mathcal{S},$ hence $H\mathcal{S}=\mathcal{S}=\big\{\{x\}:\, x\in K\big\}$.

\hypertarget{(B)}{}(B) Finally, we prove that if $K^{c}$ has infinitely
many components, then there is no such flow (even of class $C^{0}$).
Reasoning by contradiction, suppose there is such a flow. Let $\{A_{n}\}_{n\in\mathbb{N}}$
be the distinct components of $K^{c}$.

\emph{Claim. Each open invariant set $A_{n}$ contains an unstable
minimal set}:\emph{ }let $z\in A_{n}$. As $\mathbb{S}^{2}$ is compact,
$\alpha(z)\neq\emptyset$ is compact and thus contains a minimal set
$\varLambda_{n}$. Clearly, $\varLambda_{n}\subset A_{n}$, otherwise
$\varLambda_{n}\subset\mbox{bd\,}A_{n}\subset K$ which implies $\varLambda_{n}=\{x\}$,
for some $x\in K$ (by hypothesis \hyperlink{Hypothesis}{(a)}, each
$x\in K$ is an equilibrium) and $\{x\}$ is unstable (Lemma \hyperlink{Lemma 2}{2.3}),
contradicting hypothesis \hyperlink{Hypothesis}{(b)}. By hypothesis
(b), $\varLambda_{n}\not\in H\mathcal{S}$ i.e. $\varLambda_{n}\in\mbox{cl}_{H}\mathcal{U}$.
Take $\varGamma_{n}\in\mathcal{U}$ sufficiently $d_{H}$ near $\varLambda_{n}$
so that it is contained in $A_{n}$.

Now, by Blaschke Principle (\cite[p.223]{TE}), $\varGamma_{n}$ has
a subsequence $\varGamma_{n'}$ $d_{H}$ converging to some nonvoid,
compact invariant set $\varGamma\subset\mathbb{S}^{2}$. As the $A_{n}$'s
are disjoint, open and invariant, $\varGamma\subset K$. But then,
for each $x\in\varGamma\subset K$, $\{x\}\in\mbox{cl}_{H}\mathcal{U}=H\mathcal{S}^{c}$,
contradicting $\{x\}\in H\mathcal{S}$ (hypothesis (b)): this is obvious
if $\varGamma=\{x\}$, as $\mathcal{U}\ni\varGamma_{n'}\overset{d_{H}}{\longrightarrow}\varGamma$.
Otherwise $\{x\}\subsetneq\varGamma$, thus $\{x\}\in\mathcal{U}$
(Lemma \hyperlink{Lemma 2}{2.1}). \end{proof}

\selectlanguage{english}%
{\small{\bigskip{}
Centro de Matemática da Universidade do Porto }}\\
{\small{Rua do Campo Alegre, 687, 4169-007 Porto, Portugal}}\\
{\small{E-mail:}}\emph{\small{ pedro.teixeira@fc.up.pt}}\selectlanguage{british}%

\end{document}